\documentclass[11pt,leqno,british]{amsart}
\usepackage{fullpage,epsfig,graphics,amsbsy,amssymb,cancel,bbm}
\usepackage{psfrag,hyperref,color,slashed,}
\usepackage{graphicx}
\usepackage{amsfonts,mathdots,dsfont}
\usepackage{amssymb,amsmath,amscd,mathrsfs}
\usepackage{ifpdf}
\usepackage{eucal}
\usepackage{url}
\usepackage{amsthm}
\usepackage{tikz}
\usetikzlibrary{arrows,chains,matrix,positioning,scopes,snakes}

\makeatletter
\tikzset{join/.code=\tikzset{after node path={%
\ifx\tikzchainprevious\pgfutil@empty\else(\tikzchainprevious)%
edge[every join]#1(\tikzchaincurrent)\fi}}}
\makeatother

\tikzset{>=stealth',every on chain/.append style={join},
         every join/.style={->}}
\tikzset{
    >=stealth',
    punkt/.style={
           rectangle,
           rounded corners,
           draw=black, very thick,
           text width=6.5em,
           minimum height=2em,
           text centered},
    pil/.style={
           ->,
           thick,
           shorten <=2pt,
           shorten >=2pt,}
}

\newcommand{\BB}{\mathbb}

\newcommand{\bea}{\begin{eqnarray}}
\newcommand{\eea}{\end{eqnarray}}
\newcommand{\be}{\begin{equation}}
\newcommand{\ee}{\end{equation}}
\newcommand{\nn}{\nonumber}

\newcommand{\im}{\textrm{Im}\,}
\newcommand{\re}{\textrm{Re}\,}

\newcommand{\Span}{\textrm{Span}}

\newcommand{\real}{\mathbb{R}}
\newcommand{\comp}{\mathbb{C}}
\newcommand{\zahl}{\mathbb{Z}}

\newcommand{\widecheck}[1]{\stackrel{\!\!\!\vee}{#1}}
\newtheorem{theorem}{Theorem}[section]
\newtheorem{lemma}[theorem]{Lemma}
\newtheorem{proposition}[theorem]{Proposition}
\newtheorem{corollary}[theorem]{Corollary}
\newtheorem{definition}[theorem]{Definition}

\newtheoremstyle{remarkstyle}
{}%space above
{10 mm}%space below
{}
{}%
{\bfseries}
{}
{  }
{}%

\theoremstyle{remarkstyle} { \newtheorem{remark}[theorem]{Remark} }

\newenvironment{example}[1][Example]{\begin{trivlist}
\item[\hskip \labelsep {\bfseries #1}]}{\end{trivlist}}

\begin{document}

\begin{flushright}
{\tt UUITP-22/14}\\
\end{flushright}
\vspace{1cm}

\title[Multiple sine, multiple elliptic gamma functions and rational cones]%
{Multiple sine, multiple elliptic gamma functions and rational cones}

\author{Luigi Tizzano}
\address[Tizzano]{Department of Physics and Astronomy, Uppsala university,
Box 516, SE-75120 Uppsala, Sweden}
\email{luigi.tizzano@physics.uu.se}

\author{Jacob Winding}
\address[Winding]{Department of Physics and Astronomy, Uppsala university,
Box 516, SE-75120 Uppsala, Sweden}
\email{jacob.winding@physics.uu.se}
\begin{abstract} 
We define generalizations of the multiple elliptic gamma functions and the multiple sine functions, labelled by rational cones in $\mathbb{R}^r$. For $r=2,3$ we prove that the generalized multiple elliptic gamma functions enjoy a modular property determined by the cone. This generalizes the modular properties of the elliptic gamma function studied by Felder and Varchenko. The generalized multiple sine enjoy a related infinite product representation, generalizing the results of Narukawa for the ordinary multiple sine functions. 

 \end{abstract}

\maketitle
\setcounter{tocdepth}{2}
\tableofcontents

\section{Introduction}
The theory of multiple sine and multiple elliptic gamma functions already has a long history. As shown by Nishizawa \cite{N1}, it is possible to construct a hierarchy of meromorphic functions which incorporate the theta function $\theta_0(z,\tau)$ and the elliptic gamma function $\Gamma(z,\tau,\sigma)$. These functions have been called multiple elliptic gamma function $G_r(z|\tau_0,\dots ,\tau_r)$. Multiple sine functions $S_r(z|\omega_1, \dots, \omega_r)$ instead have been studied by Kurokawa \cite{K1,K2} and Jimbo-Miwa \cite{JM} starting from the Barnes' multiple gamma functions \cite{B}. A comprehensive review and thorough analysis of the properties of all these functions is contained in a paper by Narukawa \cite{Narukawa:2003}.

One of the original motivation in the Narukawa's paper was to understand the modular properties of multiple elliptic gamma functions. In fact it is very well known that the theta function exhibit modular invariance under the action of $SL_2(\BB Z) \ltimes \BB Z^2 $. For the elliptic gamma function (i.e. $G_1$ in the notation adopted here) there is a result by Felder and Varchenko \cite{FV} who proved that the modularity is instead $SL_3(\BB Z) \ltimes \BB Z^3$. The extension of these results to the case of multiple elliptic gamma functions has been proposed by Narukawa in his paper.

In this paper we define a generalization of both the multiple elliptic gamma and the multiple sine functions which we call  respectively generalized multiple gamma functions $G_r^C(z|\tau_0,\dots, \tau_r)$ and generalized multiple sine functions $S_r^C(z|\omega_1, \dots ,\omega_r)$. The superscript $C$ refers to a good rational polyhedral cone inside $\BB R^r$ which is the additional data used to generalize the functions. For general $r$ we are able to show a few properties of these new functions  that were already present for the ordinary ones. Moreover, in the specific cases $S^C_{2}$,$S^C_{3}$ and $G^C_{1}$,$G^C_{2}$ we prove an infinite product representation or modular property which is related or labelled by the cone. In particular, the infinite product representations exhibit a factorized form, with one factor from every 1-dimensional face of the cone, as well as the exponential of a generalized Bernoulli polynomial. For an $r$-dimensional cone $C$, the factors from the 1d faces are related by $SL_{r+1} (\mathbb{Z})$ elements, which are specified by the cone.
In the case of $r=1$, our generalized elliptic gamma function is closely related to the one defined in \cite{FHRZ}. A rather similar construction the one presented here is also used in \cite{conicalzeta}, but then in the context of defining a cone generalization of multiple zeta values.

The reason why we defined and studied these new generalized functions comes from physics. This is not surprising since the double sine function was already used in a paper of Jimbo and Miwa \cite{JM} to construct solutions of the quantum Knizhnik-Zamolodchikov equations. Nevertheless the physical context which motivated the present work is completely different and it is related to supersymmetric gauge theories. Recently it has been shown that the partition functions of certain supersymmetric theories can be expressed in terms of multiple sine and multiple elliptic gamma functions \cite{Hama:2011,Lockhart:2012}. In the works \cite{Qiu:2014,Qiu:2013} it has been understood that also the generalized multiple sine functions that we present here plays an important role in this picture. The appearance of the cone $C$ in physics is due to the fact that the supersymmetric gauge theories under investigation are defined on a particular class of toric manifolds and the cone is the object which classifies these types of geometries. Here we prove all the results about generalized multiple sine functions contained in \cite{Qiu:2014,Qiu:2013} and we obtain new results about properties of generalized multiple gamma functions. 

This paper is organized as follows: In section \ref{sec:multiplesine} we review known properties of multiple sine function, q-shifted factorials and multiple elliptic gamma function. In section \ref{sec:cones} we define the notion of  polyhedral cone $C$ and describe the $1$-Gorenstein condition for a cone. The section provide a technical lemma about subdivision of $2$d cones that will be crucial for proving the factorization results. In section \ref{sec:genBernoulli} we introduce the concept of generalized Bernoulli polynomials and we give an explicit expression for $B^C_{2,2}$. In section \ref{sec:generalizedsine} we define the generalized multiple sine functions, we discuss their properties and finally we prove the infinite product representation form of $S^C_2$ and $S^C_3$. Section \ref{sec:generalizedgamma} is the final section of this paper. Here we define the generalized multiple elliptic gamma functions and we prove the factorization properties of $G^C_1$ and $G^C_2$.  The case of general $r$ is the subject of ongoing work \cite{future_work}. 
\\\\
{\bf Acknowledgments:} The authors wish to express their gratitude to Jian Qiu and Maxim Zabzine for suggesting this project, as well as many interesting and useful discussions and suggestions. Both authors are supported in part by Vetenskapsr\aa det under grants \#2011-5079 and \#2014-5517, in part by the STINT grant and in part by the Knut and Alice Wallenberg Foundation.

\section{The multiple sine functions and multiple elliptic gamma functions}\label{sec:multiplesine}
In this section we review the multiple elliptic gamma functions $G_r (z|\bar \tau)$ and their relation to the multiple sine functions $S_r ( z| \bar \omega)$ as defined by Narukawa \cite{Narukawa:2003}. 
\subsection{The q-shifted factorial}
Let $x=e^{2 \pi iz}, q_j=e^{2 \pi i \omega_j}$ for
$z \in \comp$ and $\omega_j \in \comp-\real \ (0 \le j \le r)$, and
\be
\begin{split}
\underline{q} &= ( q_0, \cdots, q_r ), \\
\underline{q}^- (j) &= ( q_0, \cdots, \widecheck{q_j}, \cdots, q_r), \\
\underline{q} [j] &= ( q_0, \cdots, q_j^{-1}, \cdots, q_r), \\
\underline{q}^{-1} &= ( q_0^{-1}, \cdots, q_r^{-1} ) , 
\end{split}
\ee
where $\widecheck{q_j}$ means the exclusion of $q_j$.
When $\im\omega_j >0$ for all $j$, define the \emph{$q$-shifted factorial}
\be
(x|\underline{q})_\infty
= \prod_{j_0,\cdots,j_r=0}^{\infty}
(1- x q_0^{j_0} \cdots q_r^{j_r}).
\ee
This infinite product converges absolutely
when $|q_j| <1$.
Thus this function is a holomorphic function
with regard to $z$, whose zeros are
\be
z = \omega_0 \zahl_{\le 0} +\cdots +\omega_r \zahl_{\le 0} +\zahl.
\ee

In general we can define the $q$-shifted factorial
for $\omega_j \in \comp-\real$ as follows:
When $\im\omega_0,\cdots,\im\omega_{k-1} <0$ and
$\im\omega_k,\cdots,\im\omega_r >0$,
that is, $|q_0|,\cdots,|q_{k-1}| > 1$ and
$|q_k|,\cdots,|q_r| < 1$, we define
\be\label{qfac}
\begin{split}
(x|\underline{q})_\infty &= 
\left\{ (q_0^{-1} \cdots q_{k-1}^{-1} x|
( q_0^{-1}, \cdots, q_{k-1}^{-1}, q_k, \cdots, q_r)
)_\infty \right\}^{(-1)^k}\\
&= \left\{ \prod_{j_0,\cdots,j_r=0}^{\infty}
(1- x q_0^{-j_0 -1} \cdots q_{k-1}^{-j_{k-1} -1}
 q_k^{j_k} \cdots q_r^{j_r}) \right\}^{(-1)^k}. 
 \end{split}
\ee

The q-shifted factorial is a meromorphic function of $z$ satisfying the following functional
equations.
\begin{proposition}\label{prop:qfactorialprop}
\be \label{eq:blockinversion}
(x|\underline{q})_\infty
= \frac{1}{(q_j^{-1} x|\underline{q} [j])_\infty}, \ \ \ \ (q_j x| \underline q )_\infty = \frac{(x|\underline q)_\infty}{(x|\underline q^{-}(j))_\infty } . 
\ee

\end{proposition}
\begin{proof}
See Narukawa \cite{Narukawa:2003}
\end{proof}

\begin{proposition}\label{prop:gluing}
\be
 \frac{(x|q_0,q_1,\ldots,q_r)_{\infty}}{(x|q_0,q_0q_1,\ldots,q_r)_{\infty}}=(x|q_1^{-1},q_0q_1,\ldots,q_r)_{\infty}^{-1}~.\label{modular_new}
\ee
\end{proposition}
\begin{proof}
  We prove the lemma case by case, first let $|q_i|<1, \ i=0,\ldots,r$, then
  \bea
    \frac{(x|q_0,q_1,\cdots,q_r)_{\infty}}{(x|q_0,q_0q_1,\cdots, q_r)_{\infty}}&=&\frac{\prod\limits_{i_0,\ldots,i_r\geq0}(1-xq_0^{i_0}q_1^{i_1} \cdots q_r^{i_r} )}{\prod\limits_{i_0,\ldots,i_r\geq0}(1-xq_0^{i_0}(q_0q_1)^{i_1} \cdots q_r^{i_r})}
 =\prod\limits_{i_0 \geq 0,\; i_1>i_0, \; i_{2},\cdots,i_r \geq 0}(1-xq_0^iq_1^j q_2^{i_2} \cdots q_r^{i_r} )\nn \\
 &=&\prod\limits_{i_0,\ldots, i_r \geq 0}(1-xq_1(q_0q_1)^{i_0} q_1^{i_1} q_2^{i_2} \cdots q_r^{i_r} )\\
 &=&(xq_1|q_0q_1,q_1,\cdots, q_r)_{\infty}=(x|q_0q_1,q_1^{-1},\cdots,q_r)_{\infty}^{-1}~.\nn
 \eea
  If instead $|q_0|<1$, $|q_1|>1$ but $|q_0q_1|<1$ and $|q_2|,\cdots,|q_r| < 1$, then
    \bea
     \frac{(x|q_0,q_1,\cdots,q_r)_{\infty}}{(x|q_0,q_0q_1,\cdots, q_r)_{\infty}}&=&\frac{1}{\prod\limits_{i_0,\ldots,i_r \geq0}(1-xq_0^{i_0}q_1^{-i_1-1} \cdots q_{r}^{i_r})\prod\limits_{i_0,\ldots,i_r\geq0}(1-xq_0^{i_0}(q_0q_1)^{i_1} \cdots q_r^{i_r})}\nn\\
  &=&\frac{1}{\prod\limits_{i_0\geq 0,\; i_1\leq i_0,\; i_2,\ldots,i_r \geq 0}(1-xq_0^{i_0}q_1^{i_1} \cdots q_r^{i_r} )}=\frac{1}{\prod\limits_{i_0,\ldots,i_r\geq 0}(1-xq_1^{-i_0}(q_0q_1)^{i_1} q_2^{i_2}\cdots q_r^{i_r})} \nn \\
  &=&(xq_1|q_1,q_0q_1, q_2,\cdots,q_r)_{\infty}=(x|q_1^{-1},q_0q_1,q_2,\cdots,q_r)_{\infty}^{-1}.\nn
  \eea
  But if $|q_0q_1|>1$ 
    \bea
  \frac{(x|q_0,q_1,\cdots,q_r)_{\infty}}{(x|q_0,q_0q_1,\cdots,q_r)_{\infty}}&=&\frac{\prod\limits_{i_0,\ldots,i_r \geq0}(1-xq_0^{i_0}(q_0q_1)^{-i_1-1} q_2^{i_2} \cdots q_r^{i_r} )} {\prod\limits_{i_0,\ldots i_{r} \geq0}(1-xq_0^{i_0}q_1^{-i_1-1}\cdots q_r^{i_r})} =\prod_{i_0\geq0,-i_1-1\leq i<0, i_2,\ldots,i_r \geq 0}(1-xq_0^{i_0}q_1^{-i_1-1}\cdots q_r^{i_r}) \nn \\
  &=&\prod_{i_0,\ldots,i_r\geq0}(1-xq_0^{-i_0-1}q_1^{i_0+i_1+1} q_2^{i_2}\cdots q_r^{i_r})=(xq_1|q_0q_1,q_1,q_2,\cdots,q_r)_{\infty}\nn \\
  &=& (x|q_0q_1,q_1^{-1}, q_2,\cdots,q_r)_{\infty}^{-1}.\nn
  \eea
  By switching the role of $q_0,q_1,q_0q_1$ one can obtain the other cases. In the cases where some $|q_i|>1, \; i \geq 2$, then all the above manipulations go through in the very same way. 
  \end{proof}

\subsection{The multiple elliptic gamma function}
Let's introduce the notation for $\underline{\omega}$ as 
\be
\begin{split}
\underline{\omega} &= ( \omega_0, \cdots, \omega_r ), \\
\underline{\omega}^- (j)
&= ( \omega_0, \cdots, \widecheck{\omega_j}, \cdots, \omega_r ), \\
\underline{\omega} [j]
&= ( \omega_0, \cdots, -\omega_j, \cdots, \omega_r ), \\
- \underline{\omega}
&= ( -\omega_0, \cdots, -\omega_r ), \\
|\underline{\omega}| &= \omega_0 + \cdots + \omega_r \ ,
\end{split}
\ee
and define the \emph{multiple elliptic gamma function}
\be
\begin{split}
G_r (z|\underline{\omega})
&= (x^{-1} q_0 \cdots q_r |\underline{q})_\infty
\{(x|\underline{q})_\infty \}^{(-1)^r} \label{13b} \\
&= \{(x^{-1} |\underline{q}^{-1})_\infty \}^{(-1)^{r+1}}
\{(x|\underline{q})_\infty \}^{(-1)^r} . 
\end{split}
\ee
$G_r (z|\underline{\omega})$ is defined for
$\omega_j \in \comp - \real$ from the general definition of
$(x;\underline{q})_\infty$.
The hierarchy of $G_r (z|\underline{\omega})$ includes
the familiar theta function $\theta_0 (z|\omega) $ ($r=0$) and
the elliptic gamma function $\Gamma (z|\tau,\sigma)$ ($r=1$) 
which appeared in \cite{FV,R1}.

The multiple elliptic gamma function satisfy the following functional equations:
\begin{align}
G_r (z+1|\underline{\omega}) &= G_r(z|\underline{\omega}), \\
G_r (z+\omega_j|\underline{\omega})
&= G_{r-1} (z|\underline{\omega}^- (j)) \ 
G_r (z|\underline{\omega}), \label{eq:GrShift1} \\
G_r (z|\underline{\omega})
&= \frac{1}{G_r (z-\omega_j|\underline{\omega} [j])},  \\
G_r (-z|-\underline{\omega})
&= \frac{1}{G_r (z|\underline{\omega})}, \\
G_r (z|\underline{\omega}) G_r (z|\underline{\omega}[j])
&= \frac{1}{G_{r-1} (z|\underline{\omega}^- (j))}.
\end{align}
$G_r (z|\underline{\omega})$ can be expressed as the following infinite product when $\mathrm{Im}\ \omega_j > 0 \ \forall j$:
\be \label{eq:GrDef2}
G_r(z|\underline\omega) = \prod_{j_0,\dots, j_r=0}^{\infty} (1 - e^{2\pi i (z + j_0\omega_0 + \dots + j_r\omega_r)})^{(-1)^r} \cdot (1 - e^{2\pi i (|\underline \omega| - z + j_0\omega_0 + \dots + j_r\omega_r)}).
\ee

We also note the following formula given in \cite{Narukawa:2003}, valid for $\mathrm{Im}\, z > 0, \mathrm{Im}\ \tau_j > 0 \ \forall j$:
\be 	\label{eq:plethystic1}
( x | \underline q )_\infty  = \prod_{j_0,\ldots,j_r = 0}^{\infty} ( 1 - x q_0^{j_0} \cdots q_r^{j_r} ) = \exp \left ( - \sum_{n=1}^{\infty} \sum_{j_0,\ldots,j_r = 0}^{\infty} \frac{(xq_0^{j_0} \cdots q_r^{j_r} )^n }{n} \right ) \ ,
\ee
of which we will use a simple generalization to see that some infinite products are convergent and well-defined.

In \cite{Narukawa:2003} Narukawa also proved two important theorems about modularity of $G_r(z|\underline \omega)$:
\begin{theorem}[Modular properties of $G_r (z| \underline{\omega})$] \label{thm:G2modularity2}
If $r \ge 2, \im \frac{\omega_j}{\omega_k} \ne 0$,
then the multiple elliptic gamma function satisfies the identity
\be
\prod_{k=1}^r G_{r-2}
\left( \frac{z}{\omega_k} \bigg| \left(
\frac{\omega_1}{\omega_k}, \cdots,
\widecheck{\frac{\omega_k}{\omega_k}}, \cdots,
\frac{\omega_r}{\omega_k} \right) \right)
= \exp \left\{
- \frac{2 \pi i}{r!} B_{rr} (z|\underline{\omega})
\right\}.
\ee
\end{theorem}

\begin{theorem}[Modular properties of $G_r (z| \underline{\omega})$] \label{thm:G2modularity}
If $\im\omega_j \ne 0$ and $\im\frac{\omega_j}{\omega_k} \ne 0$, then
\be
\begin{split}
G_r (z| \underline{\omega})
&= \exp \left\{
\frac{2 \pi i}{(r+2)!} B_{r+2,r+2} (z|(\underline{\omega},-1))
\right\}  \\
& \times \prod_{k=0}^r G_r
\left( \frac{z}{\omega_k} \bigg| \left(
\frac{\omega_0}{\omega_k}, \cdots,
\widecheck{\frac{\omega_k}{\omega_k}}, \cdots,
\frac{\omega_r}{\omega_k},
-\frac{1}{\omega_k}
\right) \right)  \\
&= \exp \left\{
- \frac{2 \pi i}{(r+2)!} B_{r+2,r+2} (z|(\underline{\omega},1))
\right\}  \\
& \times \prod_{k=0}^r G_r
\left( -\frac{z}{\omega_k} \bigg| \left(
-\frac{\omega_0}{\omega_k}, \cdots,
-\widecheck{\frac{\omega_k}{\omega_k}}, \cdots,
-\frac{\omega_r}{\omega_k},
-\frac{1}{\omega_k}
\right) \right). 
\end{split}
\ee
\end{theorem}
In these theorems we see the appearance of the so called \emph{multiple Bernoulli polynomials}  $B_{r,n}(z|\omega_1,\ldots,\omega_r)$.
They are defined by the following expansion:
\be \label{eq:BernoulliExp}
	\frac{t^r e^{zt}}{\prod_{i=1}^r (e^{\omega_i t} - 1 ) } = \sum_{n=0}^\infty B_{r,n} ( z | \underline \omega ) \frac{t^n}{n!}.
\ee
In particular $B_{r,r}$ is a polynomial in $z$ of order $r$. 
\begin{remark}
Consider $\mathbb{R}^{r+1}_{\geq 0}$, which is a cone of dimension $r+1$, and has a set of $r+1$, 1-dimensional faces generated by the standard basis vectors $\{ e_1, \dots,e_{r+1} \}$. For each such face, generated by $e_i$, we construct a matrix $\tilde K_i = [e_i ,e_{\sigma(1)},\dots, e_{\sigma(r)}]^{-1}$, where $\sigma$ is a permutation of $1,\ldots, i-1,i+1,\ldots , r+1$ such that $\det (\tilde K_i) = 1$. 

It is convenient for our purposes to extend each $\tilde K_i$ into a matrix $K_i$ in $SL_{r+2}(\mathbb Z)$ by adjoining a $1$ at the lower right entry:
\[
	\tilde K_i   \to K_i =  \begin{pmatrix}
		\tilde K_i & 0 \\
		0 & 1 
	\end{pmatrix} \in SL_{r+2}(\mathbb Z).
\]
Let us also denote by $S$ the S-duality element in $SL_{r+2}(\mathbb{Z})$, i.e.
\be \label{eq:Smatrix}
	S = \begin{pmatrix}
		0 & \cdots & \cdots & -1 \\
		\vdots & 1 & 0 & \vdots \\
		\vdots & 0 & \ddots & \vdots \\
		1 &  \cdots & \cdots & 0
	\end{pmatrix}.	
\ee
Next, consider the following group action of $g\in SL_{r+2} (\mathbb{Z})$ on $(z | \omega_0,\ldots,\omega_r,1)$:
\be \label{eq:groupaction}
	g\cdot ( z | (\underline \omega) ) = (\frac{z}{ (g\underline \omega)_{r+1}} | \frac{(g\underline \omega)_{0}}{(g\underline \omega)_{r+1}},\ldots,\frac{(g\underline \omega)_{r}}{(g\underline \omega)_{r+1}} ) .
\ee
This action of $g$ is a linear fractional transformation, and we note that  $S$ together with $SL_{r+1}(\mathbb{Z})$ generates all of $SL_{r+2}(\mathbb{Z})$. With this machinery, the result of theorem \ref{thm:G2modularity} can be rewritten as
\be
G_r (z| \underline{\omega})
= \exp \left\{
\frac{2 \pi i}{(r+2)!} B_{r+2,r+2} (z|(\underline{\omega},-1))
\right\}  
 \times \prod_{i=0}^r ((S K_i)^*G_r)(z|\underline \omega) ,
 \ee
 where $(SK_i)^*$ acts on $G_r(z|\underline \omega)$ by transforming its arguments as above. The above formula include the modularity properties of $\theta_0$ and the elliptic gamma function studied by Felder and Varchenko in \cite{FV}, and is just a way of writing it that makes precise how $SL_{r+2}$ acts. We will see later how this directly generalizes for the generalized multiple elliptic gamma functions we will introduce. 
\end{remark}

We will also need the following property of $G_r$, which is a direct consequence of proposition \ref{prop:gluing}.
\begin{proposition} \label{prop:G2gluing}
	\be
		\frac{G_r ( z | \omega_0,\omega_1,\cdots,\omega_r)}{G_r(z|\omega_0,\omega_0+\omega_1,\cdots,\omega_r)}=\frac{1}{G_r ( z | -\omega_1,\omega_0+\omega_1,\cdots,\omega_r)} 
	\ee
\end{proposition}
\begin{proof}
One uses the definition of $G_r$ as well as proposition \ref{prop:gluing} as follows
\bea
		\frac{G_r ( z | \omega_0,\omega_1,\cdots,\omega_r)}{G_r(z|\omega_0,\omega_0+\omega_1,\cdots,\omega_r)} &=&
		 \frac{(x^{-1}q_0\cdots q_r | q_0,\cdots,q_r)_\infty (x | q_0,\cdots,q_r)_\infty^{(-1)^r}}{(x^{-1}q_0^2 q_1\cdots q_r | q_0, q_0 q_1,\cdots,q_r)_\infty (x | q_0,q_0 q_1,\cdots,q_r)_\infty^{(-1)^r}} \nn \\
		&=& \frac{(x^{-1}q_0\cdots q_r | q_0,\cdots,q_r)_\infty }{(x | q_1^{-1},q_0q_1,\cdots,q_r )_\infty^{(-1)^r}} \frac{ (x^{-1} q_0 q_1\cdots q_r | q_1,q_0q_1,\cdots,q_r)_\infty}{(x^{-1} q_0 q_1 \cdots q_r | q_0 , q_1, \cdots,q_r)_\infty} \nn \\
		&=& ( q_1 x | q_1,q_0 q_1,\cdots,q_r)_\infty^{(-1)^r} ( x^{-1} q_0q_1\cdots q_r | q_1,q_0q_1,\cdots q_r )_\infty \nn \\
		&=& G_r ( z + \omega_1 | \omega_1,\omega_0+\omega_1,\cdots,\omega_r) = G_r ( z | -\omega_1,\omega_0+\omega_1,\cdots,\omega_r)^{-1} . \nn
\eea
\end{proof}

\subsection{The multiple sine function}
Suppose that the points representing $\omega_1,\cdots,\omega_r \in \comp$ all lie  within the same half of the complex plane. Then the \emph{multiple zeta function} is defined by the series
\be
\zeta_r (s,z | \underline{\omega})
= \sum_{n_1,\dots,n_r=0}^\infty
\frac{1}{(n_1 \omega_1 + \cdots + n_r \omega_r +z)^s} ,
\ee
for $z \in \comp, \re s >r$.
%, where the exponential is rendered one-valued
This series is holomorphic in $s$ in the domain $\{ \re s >r \}$,
and it is analytically continued to $s \in \comp$.
Since it is  holomorphic at $s=0$,
we can define the \emph{multiple gamma function} by
\begin{equation}
\Gamma_r (z | \underline{\omega})
= \exp \left( \frac{\partial}{\partial s}
\zeta_r (s,z | \underline{\omega}) \Big|_{s=0} \right).
\label{30}\end{equation}
Now we define the multiple sine function by the form
\begin{equation}
S_r (z | \underline{\omega})
= \Gamma_r (z | \underline{\omega})^{-1}
\Gamma_r (|\underline{\omega}| -z | \underline{\omega})^{(-1)^r}.
\label{31}\end{equation}

The above definition of $\zeta_r (s,z | \underline{\omega})$ is
due to Barnes \cite{B}, and the definitions of
$\Gamma_r (z | \underline{\omega})$ and
$S_r (z | \underline{\omega})$ are due to Kurokawa \cite{K2}
and Jimbo-Miwa \cite{JM}.

Originally, the double sine function
$S_2 (z | \omega_1,\omega_2)$ have been studied to construct
solutions of certain equations
of mathematical physics as in \cite{JM}.  In recent years there has been renewed interest among physicist in the topic of generalized sine functions because of their appearance in the context of supersymmetric gauge theories. Namely, the result of a procedure called supersymmetric localization can be elegantly rewritten in terms of multiple sine functions \cite{Hama:2011,Lockhart:2012}. 

Multiple sine functions enjoy a number of properties that we list here.
\begin{proposition}[Properties of multiple sine function]\mbox{} \label{prop:S3properties}
\begin{itemize}
\item \textbf{Analyticity:}

For $ r $ odd the multiple sine is an entire function in $ z $, with zeros at
\[ z = \vec{n}\cdot\underline \omega\qquad (n_1,\dots,n_r \geq 1),\]
coming from $ \Gamma_r(|\underline\omega|-z|\underline\omega)^{-1} $, as well as zeros at 
\[ z = \vec{n}\cdot\underline \omega\qquad (n_1,\dots,n_r \leq 0),\]
coming from $ \Gamma_r(z|\underline\omega)^{-1} $. For even $ r $, the multiple sine is meromorphic with zeros for $ (n_1,\dots,n_r \geq 1)$ and poles for $ (n_1,\dots,n_r \leq 0)$;

\item \textbf{Difference equation:} 
\begin{equation} 
S_r(z+\omega_i |\underline\omega) = S_{r-1}(z|\underline\omega^-(i))^{-1}S_r(z|\underline \omega); 
\end{equation}

\item \textbf{Symmetries:} $S_r(z,\underline \omega)$ is invariant under permutations of the parameters $ \omega_i $. It also enjoys a reflection property:
\begin{align}S_r(z | \underline\omega) = S_r( |\underline\omega|-z | \underline\omega)^{(-1)^{r+1}};
\label{eq:reflection}\end{align}

\item \textbf{Rescaling invariance:}
\begin{equation} \label{eq:SrRescaling}
	S_r(c z | c\underline\omega) = S_r(z|\underline \omega),
\end{equation}
for any $ c\in \comp^* $;

\item \textbf{Factorization}:
 Let's set $x_k = e^{2\pi i z /\omega_k}$, $q_{jk} =e^{2\pi i \omega_j/\omega_k}$, $\underline{q_k} = (q_{1k},\dots, \widecheck{q_{kk}},\dots ,q_{rk})$ and $\underline{q_k}^{-1} = (q_{1k}^{-1},\dots, \widecheck{q^{-1}_{kk}},\dots ,q^{-1}_{rk})$.     If $r\geq2$, $\text{Im } \frac{\omega_j}{\omega_k} \neq 0$, then $S_r(z|\underline \omega)$ has the following factorization in terms of q-shifted factorials.
\be \label{eq:SrFactorization}
\begin{split} 
S_r (z|\underline{\omega})
&= \exp \left\{
(-1)^r \frac{\pi i}{r!} B_{rr} (z|\underline{\omega})
\right\}
\prod_{k=1}^{r} (x_k | \underline{q_k})_\infty \\
&= \exp \left\{
(-1)^{r-1} \frac{\pi i}{r!} B_{rr} (z|\underline{\omega})
\right\}
\prod_{k=1}^{r} (x_k^{-1} | \underline{q_k}^{-1})_\infty.
\end{split}
\ee
\end{itemize}
\end{proposition}

\section{Good cones}\label{sec:cones}
\begin{definition}
A strictly convex rational polyhedral cone $C \subset \BB{R}^n$ can be presented as 
\begin{equation}\label{cone}
C =\{r \in \BB{R}^n \,|\, r \cdot v_i \geq 0,\; i= 1,\dots, n  \} \subset \BB{R}^n , 
\end{equation}
where the $v_i$'s are also known as inward pointing normals of the $n$ faces of the cone. The rationality condition means that $v_i \in \BB{Z}^n$, and without loss of generality we assume that the $\{v_i\}$ are primitive, i.e. that $\gcd (v_i) = 1 \ \forall i$. 
\end{definition}
We also assume that the set $\{v_i\}$ is minimal, i.e. removing any of the $v_i$'s would change $C$. Note in particular that for $C$ 2-dimensional, this means that $C$ is specified by two normals. In the remainder of this paper we will refer to a strictly convex rational polyhedral cone just as a cone.

\begin{definition} Given a cone $C$, its \emph{dual cone} $C^*$ is defined by
\be
C^* = \{y \in \BB R^n | \;y\cdot u \geq 0 \;\textrm{for all}\; u \in C \} .
\ee
\end{definition}

For the purposes of this paper we assume the cone to be of the following type:

\begin{definition}
A cone $C$ is \emph{good}  if at every codimension $k$ face, the $k$-normals $v_{i_1}\dots v_{i_k}$ satisfy
\be
\Span_{\BB R}\langle v_{i_1}\dots v_{i_k}\rangle \;\cap \BB Z ^m = \Span_{\BB Z}\langle v_{i_1}\dots v_{i_k}\rangle .
\ee
\end{definition}  
It is easy to see that this condition is equivalent to the fact that $\{v_{i_1}\dots v_{i_k}\}$ can be completed into an $SL_m(\BB Z)$ matrix. The goodness condition was originally introduced by Lerman  \cite{Lerman:2001zua}  in the context of contact toric geometry. In this context, the cone appears as the image of the toric moment map, and the goodness condition corresponds to the toric space being a smooth manifold. 

\begin{definition}
A cone $C$ is said to be a \emph{1-Gorenstein cone} if there exists a primitive vector $\xi \in \BB Z^n$ such that 
\be
\xi \cdot v_i = 1\ ,
\ee
for each $i = 1,\dots, n$.
\end{definition}
In the toric geometry context, this corresponds to the associated toric manifold being Calabi-Yau.

For toric varieties, there is a well known result (given for example in section 2.6 of \cite{fulton}), that states that any toric variety has a refinement that resolves it's singularities. Because of the correspondence between cones and toric varieties, this can easily be translated into a statements about cones:
\begin{proposition}\label{prop:ConeSubdivision}
	Let $C$ be a $r$-dimensional cone. Then there exists a subdivision of $C$ into a finite set of smaller cones, $C = \bigcup_i C_i$, such that each subdivided cone $C_i$ has $r$ normals, forming a basis of $\mathbb{Z}^r$. 
\end{proposition}
In dimension 2, it is relatively straightforward to explicitly construct this subdivision, which we do in the following lemma. 
In higher dimensions, it is not in general easy to find such a subdivision. 
\begin{lemma} \label{lem:subdivision}
Let $W$ be a convex wedge inside $\mathbb{R}^2$, defined as $W = \{ x \in \mathbb{R}^2 : x\cdot v_1 \geq 0, x\cdot v_2 < 0 \}$ where the normal vectors $v_1,v_2 \in \mathbb{Z}^2$. Then any sum of the form 
\be
	f_W(\underline \omega) = \sum_{n\in W\cap \mathbb{Z}^2} f(n\cdot \underline \omega) ,
\ee
where $\underline \omega \in \mathbb{C}^2$ and $f:\mathbb{C} \rightarrow \mathbb{C}$ is such that the sum converges absolutely, can be written in the form 
\be
	f_W(\underline \omega) = \sum_{j=0}^n \sum_{m\in \mathbb{Z}^2_{\geq 0}} f( (m_1 + 1) ( \underline \omega \times u_j ) + m_2 (\underline \omega \times u_{j+1} ) ) ,
\ee
where $\{ u_i \}_{i=0}^{n+1}$ is a set of normals corresponding to lines that subdivide the wedge,such that $u_0 = v_1$ and $u_{n+1}=v_2$, and also $\det [ u_i,u_{i+1}] = u_i \times u_{i+1} = 1$. 
\end{lemma}
\begin{remark} Above, and in the rest of the paper, we adopt the notational shorthand $\det [ u,v] = u\times v$ when $u,v \in \mathbb{R}^2$. %Since we never use the 3d cross product, we hope that this does not cause any confusion. 
\end{remark}
\begin{proof}
We begin by giving an algorithm for finding the lines $\{ u_i \}$. 

We can assume that $v_1$ and $v_2$ are both primitive vectors, which means that there exist a $SL_2(\mathbb{Z})$ transformation that maps $v_1$ to $(0,1)$.  Because of the structure of the sum, and since it is absolutely convergent, one can reorganize it using $SL_2(\mathbb{Z})$ transformations. So without loss of generality, we assume that $v_1 = (0,1)$, and we consider the case when $v_2 = (-a,b)$, with $a,b > 0$ and $\gcd(a,b)=1$. The other possibilities for $v_2$ can be handled in a similar way and the same results holds. For two coprime integers, we can always find an integer solution to $a c - b d = 1$ where $a>d, b>c$. This is essentially the chinese remainder theorem and is easily proven. Geometrically, the line with the normal vector $u_N = (-d,c)$ will then have a slope less than the line specified by $v_2$, so it will divide the wedge in two parts, and that $c,d$ solves the equation means that the this new line and $v_2$ forms an $SL_2(\mathbb{Z})$ basis, i.e. $u_N\times v_2 = 1$. We can then continue the subdivision process, now considering the wedge between $u_n$ and $v_1$, and so on. This will generate $N$ wedges, and the key point is that the process will end in finitely many steps since the involved integers are getting smaller at every step. The process stops when we reach the line with normal $v_1$ and thus have a complete subdivision of the wedge. This is equivalent to Zagier's construction \cite{Z}, which uses continued fractions to achieve the same goal. 

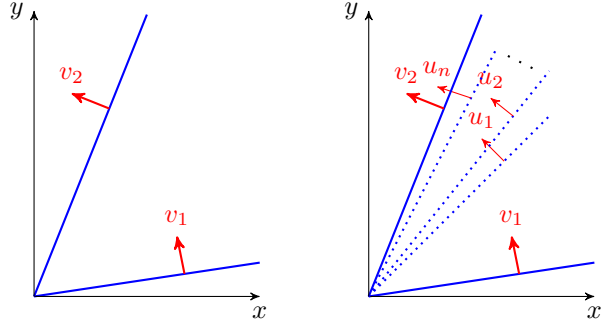
\begin{figure} \label{fig:2dconeSubdivision}

\begin{tikzpicture}[scale=1.]
\draw[thick,-,blue](0,0)-- (3,0.45);
\draw[->,black](0,0) -- (3,0) node[below]{\small $x$};
\draw[thick, ->,red] (2,0.3)-- (2-0.1,0.8) node[above] {\small$v_1$};

\draw[thick,-,blue] (0,0) -- (1.5, 3.75);
\draw[->,black](0,0) -- (0,3.8) node[left]{\small $y$};
\draw[thick, ->,red] (1, 2.5) -- (1-5*0.1,2.5+2*0.1)node[above] {\small$v_2$} ; 
\end{tikzpicture} \ \ \ \ \ 
\begin{tikzpicture}[scale=1.]

\draw[thick,-,blue](0,0)-- (3,0.45);
%\draw[thick,-,blue](0,0)-- (3,0.);
\draw[->,black](0,0) -- (3,0) node[below]{\small $x$};
\draw[thick, ->,red] (2,0.3)-- (2-0.1,0.8) node[above] {\small$v_1$};
%\draw[thick, ->,red] (2,0.0)-- (2,0.8) node[above] {\small$v_1$};

\draw[thick, dotted,blue](0,0)-- (2.4,2.4);
\draw[->,red] (1.8, 1.8) -- (1.8 - 0.3 ,1.8+ 0.3 ) node[above] {\small$u_1$};

\draw[thick,dotted,blue](0,0)-- (2.4, 3);
\draw[->,red] (1.6*1.2, 1.6*1.5)-- (1.6*1.2 - 0.1 * 3, 1.6*1.5 + 0.1*2.4) node[above] {\small$u_2$};

\draw[thick,dotted,blue](0,0)-- (1.7, 3.3);
\draw[ ->,red] (1.6*0.85, 1.6*1.65)-- (1.6*0.85 - 5*0.09 , 1.6*1.65+2*0.075) node[above] {\small$u_n$};

\draw[thick,loosely dotted, black] ( 1.85, 3.2) -- (2.25, 3.0);

\draw[thick,-,blue] (0,0) -- (1.5, 3.75);
\draw[->,black](0,0) -- (0,3.8) node[left]{\small $y$};
\draw[thick, ->,red] (1, 2.5) -- (1-5*0.1,2.5+2*0.1)node[above] {\small$v_2$} ;

\end{tikzpicture}

\caption{The left picture shows a wedge, and the right shows its subdivision into smaller $SL_2(\mathbb{Z})$ wedges, i.e. wedges where the normals satisfy $u_i \times u_{i+1} = 1$. }
\label{fig:vertexgluing} 
\end{figure}

This algorithm gives the set of normals $\{ u_i \}$. We set $u_0=v_1$ and $v_2=u_{n+1}$, the normals satisfy
\[
	\det [ u_i,u_{i+1} ] =  u_i \times u_{i+1} =  1 .
\]
This statement means that two pairwise following lines form an $SL_2(\mathbb{Z})$ basis. The sum $f_W (\omega)$ can be decomposed as a sum over each of these $n+1$ smaller wedges 
\[
	f_W(\underline \omega) = \sum_{j=0}^n \sum_{m \in W_j\cap\mathbb{Z}^2} f(m \cdot\underline\omega).
\]
Since for each wedge we are excluding the upper edge and including the lower edge, every lattice point is counted exactly once. Each of the smaller wedges can be mapped to a region over which the sum can be more easily performed, by applying the $SL_2(\mathbb{Z})$ transformation $ A_i = [ u_i,u_{i+1}]^{-1}$ to them, sending each wedge to the space $X = \{ x,y \in \mathbb{R} | x\geq 0, y< 0 \}$. This corresponds to a rearrangement of the sum, which is allowed and doesn't change its value due to absolute convergence. So we have 
\[ \begin{split}
	\sum_{m \in W_j \cap \mathbb{Z}^2 } f(m \cdot \underline\omega) &= \sum_{ m \in X \cap \mathbb{Z}^2 } f ( m \cdot (A_j\underline\omega)) = \sum_{m_1 = 0}^{\infty} \sum_{m_2 = -1}^{-\infty} f( m_1(\underline \omega \times u_{j+1} ) + m_2 (-\underline\omega \times u_{j} ) ) \\
	&=\sum_{m_1 = 0}^{\infty} \sum_{m_2 = 0}^{\infty} f( m_1(\underline \omega \times u_{j+1} ) + (m_2+1) (\underline\omega \times u_{j} ) ),
	\end{split}
\]
where we work out $A_j \underline \omega$ and recognize $\underline \omega \times u_j$, and lastly shift $m_2$ to get both summations starting from 0. This proves the lemma. 
\end{proof}
 
\subsection{Good cones and modularity} \label{sec:ConesModularity}
A good cone of dimension $r > 2$ can be seen as a way to ``label'' a set of $SL_{r+1}(\mathbb{Z})$ elements, which we describe in this section. 
Let $C$ be an $r$-dimensional good cone, with $N$ 1-dimensional faces, generated by $x_1, \ldots, x_N$. At each such face $f$, there will be $r-1$ normals of the cone, which we call $v_1^f ,\ldots,v_{r-1}^f$. We choose an ordering of these normals such that 
\[
	\det [ x_f, v_1^f ,\ldots, v_{r-1}^f ] > 0.
\]
By goodness of the cone we can also always find an $n^f$ such that 
\[
	\det [ n^f ,v_1^f ,\ldots, v_{r-1}^f ] = 1,
\]
which together with the above ordering guarantees that $x_f \cdot n^f > 0 $. 
Let us define $\tilde K_f = [ n^f, v_1^f,\ldots,v_{r-1}^f]^{-1}$, which is an $SL_r (\mathbb{Z})$ matrix. We embed this into $SL_{r+1}(\mathbb{Z})$ through
\[
	K_f = \begin{pmatrix}
		\tilde K_f & 0 \\
		0 & 1 
	\end{pmatrix},
\]
for reasons that will become clear later. Let us also denote by $S$ the S-duality element in $SL_{r+1}(\mathbb{Z})$, given explicitly in equation \eqref{eq:Smatrix}. 
%\[
%	S = \begin{pmatrix}
%		0 & \cdots & \cdots & -1 \\
%		\vdots & 1 & 0 & \vdots \\
%		\vdots & 0 & \ddots & \vdots \\
%		1 &  \cdots & \cdots & 0
%	\end{pmatrix},	
%\]
Next, consider the group action of $g\in SL_{r+1} (\mathbb{Z})$ on $(z | \omega_1,\ldots,\omega_r)$ as given in equation \eqref{eq:groupaction}, 
%thought of as coordinates in $\comp \mathbb{P}^r$:
%\be \label{eq:groupaction2}
%	g\cdot ( z | (\underline \omega, 1) ) = (\frac{z}{ (g\underline \omega)_{r+1}} | \frac{(g\underline \omega)_{1}}{(g\underline \omega)_{r+1}},\ldots,\frac{(g\underline \omega)_{r}}{(g\underline \omega)_{r+1}}, 1 ) ,
%\ee
i.e. as a linear fractional transformation.  $S$ together with $SL_{r}(\mathbb{Z})$ generates all of $SL_{r+1}(\mathbb{Z})$. So we can think of $C$ as specifying a set of $N$ $SL_{r+1}(\mathbb{Z})$ matrices, and we will see later how these matrices parametrizes the modular properties of the generalized elliptic gamma functions, as well as the infinite product representation of generalized multiple sine. Of course, the above prescription does not uniquely determine our matrices $K_f$, since there is freedom in how to choose the vector $n^f$, however one can easily check that different choices of $n^f$ only changes the parameters in $(SK_f)\cdot ( z | \underline \omega )$ by integer shifts, leaving the value of multiple elliptic gamma functions or the multiple q-factorials unchanged.

We remark that the $r=2$ case we excluded above is special, since for a 2d cone, a 1d face has only one associated normal vector $v^f$. So in this case we have no freedom in choosing an ordering, and are not guaranteed that $\det [ x_f, v^f ] > 0$. We still require that $n^f \cdot x_f > 0$, so if $\det [x_f,v^f]<0$ we will have $\det [ n^f , v^f ] = -1$.

Further, we observe that for the standard cones $\mathbb{R}^r_{\geq 0}$, following this prescription, one gets the $S$-duality matrix composed with a permutation map, of even degree (i.e. with only $+1$ as non-zero entries). So this matches precisely what is described in theorem \ref{thm:G2modularity}. 

\section{Generalized Bernoulli polynomials} \label{sec:genBernoulli}
The generating function of the original multiple Bernoulli polynomials as seen in \eqref{eq:BernoulliExp} can be written as the following expansion:
\be
	\frac{t^r e^{zt}}{\prod_{i=1}^r ( e^{\omega_i t } - 1 ) } = t^r e^{zt} \sum_{n\in \mathbb{Z}^r_{>0} } e^{-(\underline \omega \cdot n )t } , 
\ee
which converges as long as Re $\omega_i > 0$ and $t>0$. We define the \emph{generalized Bernoulli polynomials} $B_{r,n}^C$ associated to a good cone $C$ of dimension $r$ from the following generating function:
\be \label{eq:genBernoulliDef}
	t^r e^{zt} \sum_{n\in C^\circ \cap\mathbb{Z}^r} e^{-(\underline \omega \cdot n) t } = \sum_{n=0}^\infty B_{r,n}^C ( z | \underline \omega ) \frac{t^n}{n!} , 
\ee
if there exists a $c\in\comp^*$ such that $\re ( c \underline \omega ) \in (C^*)^\circ$, and for $t \in \comp$ such that the sum on the left hand side absolutely converges. 
Here $(C^*)^\circ$ denotes the interior of $C^*$. 
These polynomials will make an appearance in the following sections where we prove properties of the generalized multiple sine and elliptic gamma functions, just as the original Bernoulli polynomials appears for the normal $S_r, G_r$ functions. In particular $B_{r,r}^C$ will be a polynomial in $z$ of order $r$, just as the normal $B_{r,r}$. 

As an example, for $r=2$, we can apply the lemma \ref{lem:subdivision} to the definition and show the following:
\begin{example}
Let $C$ be a 2d cone with normals $u_0,u_{n+1}$, then
\be \label{eq:B22C}
	B_{2,2}^C ( z | \underline \omega ) = B_{2,2}(z|\underline \omega \times u_n,\underline\omega\times u_{n+1}) + \sum_{j=0}^{n-1} B_{2,2} (z+\underline \omega \times u_j | \underline \omega \times u_j,\underline \omega \times u_{j+1} ) , 
\ee	
where $\{u_j\}_{j=0}^{n+1}$ is the set of vectors subdividing $C$, as described in lemma \ref{lem:subdivision}.  
This is a straight forward application of lemma \ref{lem:subdivision}, where one treats the last term separately, removing its shift, so that all integers inside $C$ are included. 
\end{example}

The leading order term of $B_{r,r}^C$ for $r=2$ ($r=3$), i.e. the $z^2$ ($z^3$) term, is proportional to the sum of the lengths of the sides of the cone (the area of the cone) capped of by the line (plane) defined by $y\cdot \underline \omega = 1$. 
Computing the coefficient of a particular generalized Bernoulli polynomial is similar to the Ehrhart problem of counting points in $\Delta^k$, where $\Delta^k$ is the polyhedron formed by cutting of the cone by a face given by $y\cdot R = k$, where $R$ is also a rational vector \cite{Ehrhart1968}.

\section{Generalized multiple sine functions }\label{sec:generalizedsine}
Let $C$ be a good cone of dimension $r \geq 2$ as defined in section \ref{sec:cones}, and assume that $\underline \omega \in \mathbb{C}^r$ is such that for some $c\in\mathbb{C}^*$, we have $\mathrm{Re} \  c\underline \omega \in (C^*)^\circ$. 
Then we define the associated generalized multiple zeta function by the series 
\be \label{eq:zetaCdef}
	\zeta_r^C (s, z | \underline \omega ) = \sum_{ n \in C \cap \mathbb{Z}^r } \frac{ 1 } { (z + n\cdot \underline \omega )^s },
\ee
for $z \in \mathbb{C}$, $\mathrm{Re}\  s > r$. The sum is absolutely convergent in this region and the function is holomorphic in $s$ for the domain $\{ \mathrm{Re} \ s > r \}$. 
Of course for the standard cone corresponding to $\mathbb{Z}^r_{\geq 0}$ the generalized multiple zeta function becomes exactly the usual Barnes zeta function, and the requirement on $\underline \omega$ in this case means that all the points representing $\omega_1,\ldots,\omega_r$ lie on the same side of some straight line through the origin. 

The analytic properties of $\zeta^C_r$ are established trough the following result.
\begin{proposition}
Let $C$ be a cone of dimension $r$. It is always possible to write
\be
	\zeta_r^C ( s, z | \underline \omega ) = \sum_{d=1}^r \sum_{i = 0 }^{N_d} n_{d,i} \zeta_r ( s, z | \underline \omega_{d,i} ) ,
\ee
where $n_{d,i}$ are integers with $n_{r,i}=1 \ \forall i$, $\underline \omega_{d,i} \in \comp^d$ and $\zeta_0 ( s,z) = z^{-s}$. 
\end{proposition}
\begin{proof}
	We prove this by induction. The case $r=0$ is trivial, and also the $r=1$ one, since there is essentially only one cone in $\mathbb{R}$. 
			
	For $r > 1$, the argument relies on proposition \ref{prop:ConeSubdivision}, which guarantees that there always exists a subdivision of $C$ into a set of subcones $\{ C_i \}$, such that their normals form a basis of $\mathbb{Z}^r$. 
	Now, since the sum in equation \eqref{eq:zetaCdef} is absolutely convergent (for $\re s > r$)
	we can rearrange the terms of the sum without changing its value.
	 In particular this means that we can perform the sum by summing over each of the $C_i$'s and then adding up the result. 
	 In this way we will count the lattice points that lie within more than one of the $C_i$'s more than once, for which we need to compensate. 
	 The overlap between different $C_i$ cones will be given by a collection by lower dimensional cones $c_{d,i}$, with $d=0,\ldots,r-1$. A particular lower dimensional cone might overlap with several other cones, and in particular all the cones will of course overlap at the origin, $n = (0,\ldots,0)$. 
	\be
		\zeta_r^C ( s,z|\underline \omega) = \sum_i \sum_{n \in C_i \cap \mathbb{Z}^r } \frac{1}{(z+ n \cdot \underline \omega)^s } + \sum_{d=0}^{r-1} \sum_{i}^{N_d} \sum_{m \in c_{d,i}\cap\mathbb{Z}^r} n_{d,i} \frac{1}{(z+m\cdot \underline\omega)^s} .
	\ee
	Next, since all the $C_i$'s have normals that constitute a basis of $\mathbb{Z}^r$, we can use a $SL_r(\mathbb{Z})$ transformation to ``rearrange'' the sum over each $C_i$ into a sum over $\mathbb{R}^r_{\geq 0} \cap \mathbb{Z}^r = \mathbb{Z}^r_{\geq 0}$, i.e. 
	\[
	\sum_{n \in C_i \cap \mathbb{Z}^r } \frac{1}{(z+ n \cdot \underline \omega)^s } = \sum_{A_i^{-1} n \in A_i C \cap \mathbb{Z}^r } \frac{1}{(z+ (A_i^{-1} n) \cdot A_i \underline \omega)^s } =\zeta_r ( s ,z | A_i \underline \omega ) 	, 
	\]
	where $A_i \in SL_r(\mathbb{Z})$ and we recognize the ordinary $\zeta_r$. 
	 It is also clear that the sum over each lower dimensional cone $c_{d,i}$ in $\mathbb{R}^r$ can be written as a sum over lattice points inside a cone $\tilde c_{d,i} \subset \mathbb{R}^d$ , with an appropriate change of parameters $\underline \omega \rightarrow \underline \omega_{d,i}$, so we have that
	 \be
	 	\zeta_r^C ( z | \underline \omega ) = \sum_i \zeta_r ( s, z | A_i \underline \omega) + \sum_{d=0}^{r-1} \sum_{i}^{N_d} n_{d,i}\zeta^{c_{d,i}}_d ( s,z|\underline \omega_{d,i} ),
	 \ee	
	 from which the statement of the proposition follow by induction.
\end{proof}

This result is very similar in spirit to lemma \ref{lem:subdivision}, which gives us the explicit prescription for $r=2$. 
For a general cone in higher dimensions, finding explicitly the subdivision and working out a correct prescription for computing $\zeta_r^C$ is in general a difficult problem. 
However even without a general algorithm, we observe that the above result means that $\zeta_r^C$ always can be analytically continued to any $s \in \comp$ by analytic continuation of $\zeta_r$. 
And since analytic continuation is unique, the function we thus obtain does not depend on the subdivision, and one can use it to define the generalized multiple gamma function as follows:

\begin{definition} 
Given a cone $C$ the \emph{generalized multiple gamma function} associated to the cone is defined as
\be
	\Gamma_r^C ( z | \underline \omega ) = \exp \left ( \frac{\partial}{\partial s } \zeta^C_r ( s,z|\underline \omega ) \big |_{s=0} \right ) .
\ee
\end{definition}
We also define in the same way the functions $\zeta_r^{C^{\circ}}$ and $\Gamma_r^{C^\circ}$, but where the sum now is over the interior of the cone. 
%The same type of arguments again give us that $\zeta_r^{C^{\circ}}$ has an analytic continuation and thus $\Gamma_r^{C^\circ}$ is well defined. 
\begin{definition} 
Given a cone $C$ the \emph{generalized multiple sine function} associated to the cone is defined as 
\be
	S_r^C ( z | \underline \omega ) = \Gamma_r^C ( z | \underline \omega )^{-1} \Gamma_r^{C^\circ} ( - z | \underline \omega )^{(-1)^r}. 
\ee
\end{definition}
This generalizes the definition of $S_r$ first given by Kurokawa \cite{K2}. Note that for the standard cone given by $\mathbb{Z}^r_{\geq 0}$, this definition matches the ordinary $S_r$ exactly, since the restriction to the interior of the cone corresponds to a shift of $z$ by $|\underline \omega | = \sum_i^r \omega_i$. 

\subsection{Properties of generalized multiple sine function $S_r^C$}
\begin{proposition}[Rescaling invariance]
Let $C$ be either a good 2d cone ($r=2$), or a good 1-Gorenstein 3d cone ($r=3$). Then we have 
\be
	S_r^C ( cz|c\underline \omega) = S_r^C ( z | \underline \omega ),
\ee
for $c\in\comp^*$. 
\end{proposition}
\begin{proof}
This follows from the rescaling property of the usual $S_r$ given in equation \eqref{eq:SrRescaling}, combined with the representation of $S_r^C$ as a finite product of ordinary $S_r$ that we obtain in the proofs of theorems \ref{thm:S2Cfactorization} and \ref{thm:S3Cfactorization}, specifically in equations \eqref{eq:S2CinS2} and \eqref{eq:S3CinS3}. 
\end{proof}
\begin{proposition}[Analyticity]
	For $ r $ odd the generalized multiple sine is an entire function in $ z $, with zeros at
	\[ 
		z = n\cdot\underline \omega\qquad  n \in C^{\circ} \cap \mathbb{Z}^r  ,
	\]
	coming from $\Gamma^{C^\circ}_r( - z|\underline\omega)^{-1} $, as well as zeros at 
	\[ 
		z = n \cdot\underline \omega\qquad -n \in C \cap \mathbb{Z}^r,
	\]
	coming from $ \Gamma_r^C(z|\underline\omega)^{-1} $. For even $ r $, the multiple sine is meromorphic with zeros for $n \in C^{\circ} \cap \mathbb{Z}^r $ and poles for $-n \in C\cap \mathbb{Z}^r$.
\end{proposition}
\begin{proof}
	This is straightforward from the definition, applying the same logic as for the ordinary $S_r$ functions, looking at the zeros or poles of $(\Gamma_r^C)^{\pm 1}$. 
\end{proof}

Perhaps the most intriguing property of these generalized triple sine products it that they seem to have an infinite product representation, similar to the one proved for $S_r$ in \cite{Narukawa:2003}, i.e. equation \eqref{eq:SrFactorization}, and which is closely related to the geometry of the cone. For the cases we can handle, $S_r^C$ can be written as a product of q-shifted factorials, one for each 1-dimensional face of the cone, with parameters determined by the inward pointing normals determining the codimension 1 faces intersecting along this 1d face. Also appearing in this infinite product representation is an exponential of a generalized Bernoulli polynomial, introduced in section \ref{sec:genBernoulli}.
\begin{remark}
For technical reasons we require the cone to satisfy the Gorenstein condition in the $r=3$ case. This condition helps us in the following way: let $\xi$ be the vector such that $\xi \cdot v_i = 1 \ \forall i$, then for $m \in C \cap \mathbb{Z}^3$, we have $(m+\xi)\cdot v_i = m\cdot v_i + 1 \geq 1$, so we can write the sum over the interior of the cone as a sum over the whole cone but with a shift of $z$ by $\xi \cdot \underline \omega$. %This helps us to prove our results, however we expect them to hold also for general $r$ and without the Gorenstein condition.  	
\end{remark}

\subsection{Infinite product representation of $S_2^C$}

\begin{theorem}[Infinite product representation of $S_2^C$] \label{thm:S2Cfactorization}
Let $C$ be a two dimensional good cone defined by the normals $\{ v_1,v_2 \}$, and let $\underline \omega$ be such that $\im \left( \frac{\underline \omega \cdot p}{\underline \omega \cdot q} \right ) \neq 0$ for any $p,q \in \mathbb{Z}^2 \setminus \{ 0 \},\ p\neq q$. Then 
\be
	S_2^C ( z | \underline \omega) = e^{\frac{\pi i}{2} B_{22}^C ( z | \underline \omega ) } \prod_{f=1}^2  ( e^{2\pi i \frac{z}{ \tau^f_3 } } | e^{2\pi i \frac{ \tau^f_2 } { \tau^f_3 } } )_\infty , 
	% ( e^{2\pi i \beta_1 z } | e^{2\pi i \beta_1 \epsilon_1 } )_\infty( e^{2\pi i \beta_2 z } | e^{2\pi i \beta_2 \epsilon_2 } )_\infty
\ee
where
\be
	\underline \tau^f = SK_f (\underline \omega,1),
\ee
with $K_f, S \in SL_{3}^{\pm} (\mathbb{Z})$ are as described in section \ref{sec:ConesModularity}.
%\be \label{eq:S2Cparams}
%\begin{split}
%	\beta_1 = \det [ \underline \omega, v_1 ]^{-1}, \ \ \ \epsilon_1 = \det [ \underline \omega, n_1 ]  \\
%	\beta_2 = \det [v_2, \underline \omega]^{-1}, \ \ \ \epsilon_2 = \det [\underline \omega, n_2 ] ,
%\end{split}
%\ee
%where $n_i \in \mathbb{Z}^2$ is such that $\det [v_i,n] = 1$, and 
$B_{22}^C$ is a generalized Bernoulli polynomial as defined in section \ref{sec:genBernoulli}, and explicitly given by equation \eqref{eq:B22C}.
\end{theorem}

\begin{proof}
For Re $s > 2$, the sum in the definition of the generalized zeta-function will converge absolutely, so using lemma \ref{lem:subdivision} one can express $\zeta_2^C$ as
\be
	\zeta_2^C ( s,z | \underline \omega) = \zeta_2 ( s,z| \underline \omega \times u_n , \underline \omega \times u_{n+1} ) + \sum_{j=0}^{n-1} \zeta_2(s,z + \underline \omega \times u_j | \underline \omega \times u_j , \underline \omega \times u_{j+1} ) , 
\ee
where $u_0=v_1$ and $u_{n+1}=-v_2$ (the minus sign appears because of how the wedge was previously defined in lemma \ref{lem:subdivision}), and $u_1,\ldots,u_n$ are the subdivision also given in lemma \ref{lem:subdivision}. We here assume that the cone has the shape depicted in figure \ref{fig:2dconeSubdivision}, the other cases all work out in essentially the same way. In the above formula, we separate the term from the `last' ($u_n,u_{n+1}$) wedge since we have to include both boundaries of $C$ in order not to miss the points along the $v_2$ face.  
Similarly, for the generalized zeta function associated to the interior of the cone, $C^\circ$, we use the opposite choice of what to include in every subdivided wedge (i.e. we consider $\tilde W_j = \{ n\cdot u_j > 0, n\cdot u_{j+1} \leq 0 , n\in\mathbb{Z}^2 \}$ instead of $W_j$ as defined in the lemma). Then:
\[
	\zeta_2^{C^\circ} ( s,z| \underline \omega ) = \zeta_2(s,z+\underline \omega \times ( u_n + u_{n+1}) | \underline \omega \times u_n, \underline \omega \times u_{n+1} ) + \sum_{j=0}^{n-1} \zeta_2(s,z + \underline \omega \times u_{j+1} | \underline \omega \times u_j , \underline \omega \times u_{j+1} ) ,
\]
Going through the construction of $S^C_r$ and remembering the definition of the usual $S_2$, we can write: 
\be \label{eq:S2CinS2}
S_2^C ( z | \underline \omega ) = S_2 ( z |  \underline \omega \times u_n , \underline \omega \times u_{n+1} ) \prod_{j=0}^{n-1} S_2 ( z + \underline \omega \times u_j | \underline \omega \times u_j , \underline \omega \times u_{j+1} ).
\ee
Next, using equation \eqref{eq:SrFactorization}, which applies as guaranteed by the condition on $\underline \omega$, it follows that: 
\be
\begin{split}
	S_2^C ( z | \underline \omega ) =  e^{\frac{\pi i}{2} B_{2,2}^C(z|\underline \omega)} 
	(e^{2\pi i \frac{z}{\underline \omega\times u_n}} | e^{2\pi i \frac{\underline \omega \times u_{n+1}}{\underline \omega \times u_n} })_\infty 
	(e^{2\pi i \frac{z}{\underline \omega\times u_{n+1}}} | e^{2\pi i \frac{\underline \omega \times u_{n}}{\underline \omega \times u_{n+1}} })_\infty  \\
	\times \prod_{j=0}^{n-1} (e^{2\pi i \frac{z +  \underline \omega \times u_j }{\underline \omega\times u_j} }|e^{2\pi i  \frac{\underline \omega \times u_{j+1}}{\underline \omega \times u_j} })_\infty 
	(e^{2\pi i \frac{z + \underline \omega \times u_j }{\underline \omega\times u_{j+1}}} |e^{2\pi i  \frac{\underline \omega \times u_{j}}{\underline \omega \times u_{j+1} } })_\infty ,
\end{split}
\ee	
where we've collected the Bernoulli polynomials into the $B_{2,2}^C$ of equation \eqref{eq:B22C}.
Next, consider two factors of this product with the same denominator in their argument of the q-factorial, i.e. 
\[
	\star = (e^{2\pi i  \frac{z + \underline \omega \times u_j}{\underline \omega \times u_j } }|e^{2\pi i  \frac{\underline \omega\times u_{j+1}}{\underline \omega \times u_{j} } })_\infty 
	( e^{2\pi i \frac{z + \underline \omega \times u_{j-1}}{\underline \omega \times u_{j} }} | e^{2\pi i \frac{\underline \omega\times u_{j-1}}{\underline \omega \times u_{j} } })_\infty . 
\]
Because of multiple q-factorials properties, as well as the fact that $u_{i-1} + u_{i+1} = \mathbb{Z} u_i$, those two factors cancels:
\bea
	\star &=& ( e^{2\pi i \frac{z}{\underline \omega \times u_j }} | e^{2\pi i \frac{\underline \omega\times u_{j+1}}{\underline \omega \times u_{j} }} )_\infty 
	( e^{2\pi i \frac{z }{\underline \omega \times u_{j} } }| e^{- 2\pi i \frac{\underline \omega\times u_{j-1}}{\underline \omega \times u_{j} } })^{-1}_\infty \nn \\
	&=& (e^{2\pi i  \frac{z}{\underline \omega \times u_j } }| e^{2\pi i  \frac{\underline \omega\times u_{j+1}}{\underline \omega \times u_{j} } })_\infty 
	( e^{2\pi i \frac{z }{\underline \omega \times u_{j} }} |e^{- 2\pi i \frac{\underline \omega \times(\mathbb{Z}u_j - u_{j+1})}{\underline \omega \times u_{j} }} )^{-1}_\infty \nn \\
	&=& 1 \ . \nn
\eea
So each of these blocks will in fact cancel, except the two which do not pair up with any other block, i.e. the 2 q-factorials
\be
	( e^{2\pi i \frac{z}{\underline \omega \times u_0 } }| e^{2\pi i \frac{\underline \omega \times u_1}{\underline \omega \times u_0 }} )_\infty 
	( e^{2\pi i \frac{z}{\underline \omega \times u_{n+1} } }| e^{2\pi i  \frac{\underline \omega \times u_{n}}{\underline \omega \times u_{n+1} } })_\infty  .
\ee
Finally, using the definitions of $u_0, u_1, u_n$ and $u_{n+1}$ we have that
\be
	\underline \omega \times u_0 = \det [ \underline \omega, v_1 ] , \ \ \ \underline \omega \times u_1 = \det [\underline \omega, n_1 ] , 
\ee
since $n_1$ is defined as a vector that satisfies $\det[ v_1 , n_1]=1$ and $u_1$ fulfills this by construction. Similarly,
\be
	\underline \omega \times u_{n+1} = \det [\underline \omega, -v_2 ] = \det[ v_2, \underline \omega]  , \ \ \ \underline \omega \times u_n = \det [ \underline \omega,  n_2 ] .
\ee
Now, we only need to compare this with the statement in the theorem. For face 1 (for this shape of the cone), corresponding to $v_1$, we have $\det [x_1, v_1 ] > 0$ and thus $\tilde K_1 = [n_1,v_1]^{-1} \in SL_{2}(\mathbb{Z})$. And it is simple to work out that indeed
\be
	\tilde K_1 \begin{pmatrix} \omega_1 \\ \omega_2 \end{pmatrix} = \begin{pmatrix} \det [ \underline \omega, v_1] \\ \det[ \underline \omega, n_1] ,\end{pmatrix}
\ee
which after extending and composing with $S$, to $S K_f \in SL_3(\mathbb{Z})$, gives the right answer. It is also easy to realize that for face 2, corresponding to $v_2$, $\det[x_2,v_2]<0$ and thus we should consider $\tilde K_f = [-n_2,v_2]$, which has determinant $-1$. Comparing in the same way as above, we again see that it matches. 
\end{proof}

\subsection{Infinite product representation of $S_3^C$}

\begin{theorem}[Infinite product representation of $S_3^C$]
\label{thm:S3Cfactorization}
	Let $C$ be a good three dimensional cone, satisfying the 1-Gorenstein condition, with inwards pointing normals $\{ v_i \}$, $i=1,\ldots,N$, a set of 1d faces $\Delta_1^C$, and let $\underline \omega$ be such that $\im \left ( \frac{\underline \omega \cdot p}{\underline \omega \cdot q} \right ) \neq 0$ for $p,q \in \mathbb{Z}^3 \setminus {0}$ and $p\neq q$. Then the associated generalized triple sine $S_3^C ( z | \underline \omega)$ can be written in the following form
	\be
		S_3^C ( z | \underline \omega) = e^{ - \frac{\pi i}{6}B_{33}^C (z | \underline \omega) } \prod_{f \in \Delta_1^C} ( e^{2\pi i \frac{ z } { \tau^f_4 } } | e^{2\pi i \frac{\tau^f_2}{\tau^f_4}}, e^{2\pi i \frac{\tau^f_3}{\tau^f_4} })_\infty , 
	\ee
	where
	\[
		\underline \tau^f = S K_f  (\underline \omega, 1 ) ,
	\]
	with $S$ and $K_f$ as defined in section \ref{sec:ConesModularity}. 
\end{theorem}	

Before giving the proof, we first show some preliminary results about about products of q-factorials over wedges which will be used in the proof of theorem \ref{thm:S3Cfactorization}. 

	\begin{proposition} \label{prop:induction}
		Consider two rays from the origin in $\mathbb{R}^2$ with primitive normals $v,w \in \mathbb{Z}^2$ pointing in an anti-clockwise direction such that $v\times w = 1$. Let $u_1,\ldots,u_n$ denote the normals of $n$ extra lines subdividing the wedge defined by $v,w$ such that $v \times u_1 = u_i \times u_{i+1} = u_n\times w = 1$, and let $u_0 = v, u_{n+1} = w$. Then 
		\be
			\prod_{i=0}^{n} ( e^{2\pi i z } | e^{2\pi i \underline \omega \times u_i }, e^{ - 2\pi i \omega \times u_{i+1} } )_\infty = (e^{2\pi i z } | e^{2\pi i \underline \omega \times v}, e^{-2\pi i \underline \omega \times w } )_\infty  . 
		\ee
	\end{proposition}
\begin{proof}
	We proof this by induction on the number of lines added. Up to an $SL_2(\mathbb{Z})$ transformation, we can assume that $v= (0,1)$ and $w = (-1,0)$, and then we first of all consider the case of adding a single line $u$. The only possibility is $u=v+w=(-1,1)$ and for these lines, the expression above is exactly proposition \ref{prop:gluing}.
	Next, consider adding $n$ redundant lines with normals $u_1,\ldots,u_n$, still assuming $v=(0,1), w=(-1,0)$. If one of the added lines has $u_i =u=(-1,1)$, we instead consider the two wedges $v,u_i$ and $u_i, w$, each of which satisfies the conditions of the proposition, and will be subdivided by less than $n$ lines and the proof follows by induction. This relies on the assumption that it is always possible to find a line with normal $u$ in the subdivision proposed here.  
	
	To prove this, assume first that all $n$ lines are between the lines with normals $v$ and $u$. Then the last of them must have normal $u$. If instead all lines are between the lines with normals $u$ and $w$, then the first of them must have normal $u$. The remaining case is that there are lines between both  the lines with normals $u,v$ and $u,w$. Assume that none of the lines have normal equal to $u$, then there are two primitive normals $u_i,u_{i+1}$ on either side of $u$. Let $u_i = (-c,d)$ and $u_{i+1} = (-a,b)$, with $a,b,c,d > 0$. In order for the corresponding lines to be on either side of the $u$-line, these integers needs to satisfy $\frac{c}{d} < 1 < \frac{a}{b}$ or $d>c \geq 1$ and $a > b \geq 1$. However this leads to $u_i \times u_{i+1} = ad-bc > 1$ which is in contradiction with the assumptions made. So we can always find the line $u$ among the lines of the subdivision.
		
	\end{proof}

\begin{corollary} \label{cor:A1cancellation}
	Assume that there exist a set of normals $\{ u_i \}_{i=0}^{n+1}$ satisfying $u_i \times u_{i+1} = 1$ that covers the entire $\mathbb{R}^2$, i.e. $u_0 = u_{n+1}$. Then 
	\be \label{eq:A1canc}
		\prod_{i=0}^{n} ( e^{2\pi i z } | e^{2\pi i \underline \omega \times u_i }, e^{ - 2\pi i \omega \times u_{i+1} } )_\infty = 1 - e^{2 \pi i z }.
	\ee
\end{corollary}
 \begin{proof} 
 Using a $SL_2 (\mathbb{Z})$ transform, we can put $u_0, u_1$ into $(0,1)$ and $(-1,0)$ respectively. Then by the argument used in the proof of proposition \ref{prop:induction}, the normal $(1,-1)$ must be among the rest of the normals. Thus, after the $SL_2 (\mathbb{Z})$ transformation, we can view the normals as giving some subdivision of the two wedges specified by the normals $(-1,0),(1,-1)$ and $(1,-1),(0,1)$, and by proposition \ref{prop:induction} we know that such subdivisions do not change the value of the product. The final step is computing the above product for the three normals $(0,1),(-1,0),(1,-1)$. This is easily done using properties of the multiple q-factorials,
	 \[
		\prod_{i=1}^3  ( e^{2\pi i z } | e^{2\pi i \underline \omega \times v_i }, e^{ - 2\pi i \omega \times v_{i+1} } )_\infty = 1 - e^{2 \pi i z },
	\]
proving the corollary. 
  \end{proof}

\begin{proof}[Proof of theorem \ref{thm:S3Cfactorization}]
	Since the cone satisfies the 1-Gorenstein condition, we choose coordinates such that for all normals, their first component is equal to 1, i.e. the normal vectors can be written as $v_i = (1, -L_i)$ for two-vectors $L_i= (L_i^1,L_i^2)$. This is equivalent to having a coordinate system where $\xi = (1,0,0)$. A vector $n = (n_1,n_2,n_3)$ inside the cone can be written as 
	\be
		n_1 \geq L_i^2 n_2 + L_i^3 n_3 , \ \ \ i=1,\ldots,N. 
	\ee	
	At each point in the $(n_2,n_3)$ plane, one of those inequalities will dominate and give the true lower bound of $n_1$. Geometrically those regions correspond to the regions over the different faces of the cone. Let $W_i$ be the wedge of the $n_2-n_3$ plane corresponding to face $i$, $W_i = \{ m \in \mathbb{R}^2 : m\cdot v_i \geq 0 , m\cdot v_{i+1} < 0 \}$, and note that for $\re s$ large enough, the sum is absolutely convergent so that we can rearrange the order of summation without changing its value. Using this, the generalized zeta function $\zeta^C_3 ( z | \underline \omega )$ is expressed as
	\be
		\zeta_3^C(s,z|\underline \omega) = \sum_{i = 1}^N \sum_{(n_2,n_3) \in W_i} \sum_{n_1 = L_i^2 n_2 + L_i^3 n_3}^{\infty} \frac{1}{(z + n\cdot \underline \omega )^s }   + \sum_{n_1 = 0}^{\infty } \frac{1}{(z + n_1 \omega_1 )^s } , 
	\ee
	where the last sum is performed over the origin of the $n_2-n_3$ plane in order not to miss the lattice points there. 
	Next, recognizing the last term as $\zeta_1$ gives 
	\[
	 \begin{split}
		\zeta_3^C(s,z|\underline \omega) &=  \sum_{i = 1}^N \sum_{(n_2,n_3) \in W_i} \sum_{n_1 = 0}^{\infty} \frac{1}{(z + n_1 \omega_1 + (\omega_2 + L_i^2 \omega_1 ) n_2 + (\omega_3 + L_i^3 \omega_1) n_3 )^s }  + \zeta_1 ( s,z|\omega_1)  \\
		&=  \sum_{i = 1}^N  \sum_{n_1 = 0}^{\infty} \sum_{m \in W_i} \frac{1}{(z + n_1 \omega_1 + m \cdot \underline \omega_i )^s }  + \zeta_1 ( s,z|\omega_1),
	\end{split}
	\]
	where we define $\underline \omega_i = (\omega_2 + L_i^2\omega_1, \omega_3 + L_i^3 \omega_1 )$.	
	
	The absolute convergence of the zeta function (for large enough $\re s$) allows us to apply lemma \ref{lem:subdivision} to the sum. This gives a set of subdividing lines $\{ u_{i,k} \}_{k=0}^{N_i+1}$ for each wedge labeled by $i$. Each subdivided wedge bounded by $u_{i,j},u_{i,j+1}$ contributes with a copy of $\zeta_2$. 
	\[ \begin{split} 
		\zeta_3^C(s,z|\underline \omega) &= \zeta_1 ( s,z|\omega_1 ) +  \sum_{i = 1}^N  \sum_{n_1 = 0}^{\infty} \sum_{j=0}^{N_i} \zeta_2 ( s,z + \underline \omega_i \times u_{i,j} + n_1 \omega_1 | \underline \omega_i \times u_{i,j}, \underline \omega_i \times u_{i,j+1} ) \\
		&= \zeta_1 ( s,z| \omega_1 ) + \sum_{i=1}^N \sum_{j=0}^{N_i} \zeta_3 ( s, z +  \underline \omega_i \times u_{i,j}  | \omega_1,  \underline \omega_i \times u_{i,j}, \underline \omega_i \times u_{i,j+1} ). 
	\end{split} \] 
	In this way $\zeta^C_3$ has been resolved into a finite sum of ordinary $\zeta_3$ functions. This directly implies the factorization of the $\Gamma_3^C$: 
	\[  
		\Gamma_3^C ( z | \underline \omega ) = \Gamma_1 ( z | \omega_1) \prod_{i=1}^N \prod_{j=0}^{N_i} \Gamma_3 (z + \underline \omega_i \times u_{i,j}  | \omega_1,  \underline \omega_i \times u_{i,j}, \underline \omega_i \times u_{i,j+1} ),
	\]
	In the definition of the wedges $W_i$, the choice of including the lower edge (i.e. $m\cdot v_i \geq 0$) while excluding the upper edge ($m\cdot v_{i+1} < 0$) is arbitrary. One can equally well define the wedges with the opposite convention, this indeed leads to the alternative representation of $\Gamma_3^C$ as 
	\[
		\Gamma_3^C ( z | \underline \omega ) =  \Gamma_1 ( z | \omega_1) \prod_{i=1}^N \prod_{j=0}^{N_i} \Gamma_3 (z + \underline \omega_i \times u_{i,j+1}  | \omega_1,  \underline \omega_i \times u_{i,j}, \underline \omega_i \times u_{i,j+1} ).
	\]
	The 1-Gorenstein condition lets us write the sum over the interior of the cone as a shift of $z$ by $\xi \cdot \underline \omega = \omega_1$ since $\xi = (1,0,0)$, and using this as well as both the above representations of $\Gamma_3^C$, we write the generalized triple sine as 
	\begin{align*}
		S_3^C ( z | \underline \omega ) &=  \Gamma_3^C ( z | \underline \omega )^{-1} \Gamma_3^C ( \xi \cdot \underline \omega - z | \underline \omega )^{-1}\nn \\
		&= \Gamma_1 ( z | \omega_1)^{-1} \Gamma_1 (\xi \cdot \underline \omega - z | \omega_1)^{-1} \prod_{i=1}^N \prod_{j=0}^{N_i} \Gamma_3 (z + \underline \omega_i \times u_{i,j}  | \omega_1,  \underline \omega_i \times u_{i,j}, \underline \omega_i \times u_{i,j+1} )^{-1} \nn \\
		&\times \Gamma_3 (\omega_1 - z + \underline \omega_i \times u_{i,j+1}  | \omega_1,  \underline \omega_i \times u_{i,j}, \underline \omega_i \times u_{i,j+1} )^{-1}.
	\end{align*}
	The factors of $\Gamma_3$ combines into the usual triple sine, and the factors of $\Gamma_1$ combines into $S_1$ which is nothing but  the usual $\sin$ function
	\be \label{eq:S3CinS3}
		S^C_3 ( z | \underline \omega ) = 2  \sin \left ( \frac{\pi z }{\omega_1} \right ) \prod_{i=1}^N \prod_{j=0}^{N_i} S_3 (z + \underline \omega_i \times u_{i,j}  | \omega_1,  \underline \omega_i \times u_{i,j}, \underline \omega_i \times u_{i,j+1} ).
	\ee
	Next, because of the condition on $\underline \omega$, one can now use the infinite product representation of $S_3$ given by equation  \eqref{eq:SrFactorization}, and expand this as 
	\begin{equation} \label{eq:usingfactor} \begin{split}
		&S^C_3 ( z | \underline \omega ) =  2  \sin \left ( \frac{\pi z }{\omega_1} \right ) \times B \times \overbrace{\prod_{i=1}^N \prod_{j=0}^{N_i} ( e^{2\pi i \frac{ z + \underline \omega_i \times u_{i,j} } {\omega_1} } | e^{2\pi i \frac{\underline \omega_i \times u_{i,j}}{\omega_1}}, e^{2\pi i \frac{\underline \omega_i \times u_{i,j+1}}{\omega_1}} )_\infty}^{A_1} \times \\
		& \underbrace{ \prod_{i=1}^N \prod_{j=0}^{N_i}  ( e^{2\pi i \frac{ z } { \underline \omega_i \times u_{i,j} } } | e^{2\pi i \frac{\omega_1}{ \underline \omega_i \times u_{i,j} }}, e^{2\pi i \frac{\underline \omega_i \times u_{i,j+1}}{ \underline \omega_i \times u_{i,j} }} )_\infty  ( e^{2\pi i \frac{ z + \underline \omega_i \times u_{i,j} } {\underline \omega_i \times u_{i,j+1}} } | e^{2\pi i \frac{\underline \omega_i \times u_{i,j}}{\underline \omega_i \times u_{i,j+1}}}, e^{2\pi i \frac{\omega_1}{\underline \omega_i \times u_{i,j+1}}} )_\infty}_{A_2} , 
		\end{split}
	\end{equation}
	where
	\be
	 	B = \prod_{i=1}^N \prod_{j=0}^{N_i} e^{ - \frac{\pi i}{6}  B_{33} (z + \underline \omega_i \times u_{i,j}  | \omega_1,  \underline \omega_i \times u_{i,j}, \underline \omega_i \times u_{i,j+1} )} . 
	\ee
	We will deal with the factors $A_1,A_2$ and $B$ separately. The factor $A_1$ can be written exactly in the form of the product considered in corollary \ref{cor:A1cancellation}, using the property \eqref{eq:blockinversion} of q-factorials, and applying the corollary gives 
	\be
		A_1 = (1 - e^{2\pi i \frac{z}{\omega_1}})^{-1}. 
	\ee
	The $A_1$ factor almost cancels against the $2\sin(\pi z/\omega_1)$ and we are left with $ie^{-i\pi \frac{z}{\omega_1}}$, this term will be considered together with the $B$ factor.

	To deal with the $A_2$ factor, consider first the contribution from two neighboring small wedges inside the same overall wedge, i.e. the contribution from the wedge between $u_{i,k-1}, u_{i,k}$ and the wedge $u_{i,k},u_{i,k+1}$. To prevent notational clutter, we suppress the index $i$ on $u_{i,k}$ in the next few formulas. By the definition of the subdivision we have that 
	\be \label{eq:sumofnormals}
		u_{k-1} \times u_{k} = u_{k}\times u_{k+1} = 1 \ \ \Rightarrow \ \  u_{k-1} + u_{k+1} = a_k u_{k}, \ \ \ \ a_k \in \mathbb{Z}.
	\ee
	Then, we consider the contributions from these two wedges that have the same denominator inside their exponentials, i.e. 
	\be \label{eq:blockcancellation} \begin{split}
		&( e^{2\pi i \frac{z }{\underline \omega_i \times u_k } } | e^{2\pi i \frac{\omega_1}{\underline \omega_i \times u_k } }, e^{2\pi i \frac{\underline\omega_i \times u_{k+1} }{\underline \omega_i \times u_k } } )_\infty 
		( e^{2\pi i \frac{z + \underline \omega_i \times u_{k-1} }{\underline \omega_i \times u_k } } | e^{2\pi i \frac{\omega_1}{\underline \omega_i \times u_k } }, e^{2\pi i \frac{ \underline \omega_i \times u_{k-1} }{\underline \omega_i \times u_{k} } } )_\infty \\
		&= ( e^{2\pi i \frac{z}{\underline \omega_i \times u_k } } | e^{2\pi i \frac{\omega_1}{\underline \omega_i \times u_k } }, e^{2\pi i \frac{\underline\omega_i \times u_{k+1} }{\underline \omega_i \times u_k } } )_\infty
		 ( e^{2\pi i \frac{z}{\underline \omega_i \times u_k } } | e^{2\pi i \frac{\omega_1}{\underline \omega_i \times u_k } }, e^{-2\pi i \frac{ \underline \omega_i \times u_{k-1} }{\underline \omega_i \times u_{k} } } )^{-1}_\infty \\
		&=  ( e^{2\pi i \frac{z}{\underline \omega_i \times u_k } } | e^{2\pi i \frac{\omega_1}{\underline \omega_i \times u_k } }, e^{2\pi i \frac{\underline\omega_i \times u_{k+1} }{\underline \omega_i \times u_k } } )_\infty
		 ( e^{2\pi i \frac{z}{\underline \omega_i \times u_k } } | e^{2\pi i \frac{\omega_1}{\underline \omega_i \times u_k } }, e^{2\pi i \frac{\underline \omega_i \times (u_{k+1} - a_k u_k ) }{\underline \omega_i \times u_{k} } } )^{-1}_\infty \\
		&= 1.
	\end{split} 
	\ee
	where one makes use the periodicity and inversion properties of the q-shifted factorial, as well as the relation \eqref{eq:sumofnormals}. This means that almost every factor in $A_2$ cancels, and the only ones that are left are the ones coming from the edge of the large wedges, i.e. the ones bounded by either $u_{i,0},u_{i,1}$ or $u_{i,N_i}, u_{i,N_i+1}$.

	Let's consider such a case, where for convenience we call the involved lines $u_0,u_1$ and $u_2$ and let $u_1$ be the boundary between the two faces. Let also $\underline \omega_1, \underline \omega_2$ be the two corresponding $\underline \omega_i$-vectors. The faces which the line $u_1$ divides have normals $v_1 =(1,-L_1)$ and $v_2 =(1, -L_2)$. Since $u_1$ is normal to the line along the intersection of these two faces,
	$u_1$ is parallel to $L_1 - L_2$. From the goodness of the cone, the vector $L_1 - L_2$ is primitive, as is $u_1$, and hence $u_1 = L_2 - L_1$, where the right sign is fixed by the convention that the normals $u_i$ point counter clockwise. This relation holds for any line that separates two wedges $W_i,W_{i+1}$, and it implies that 
	\be \label{eq:omegaparamrel}
		\underline \omega_1 \times u_1 = \underline \omega_2 \times u_2.
	\ee
	Also from the definition of $\underline \omega_i$ we have from $u_1 = L_2 - L_1$ that
	\be \label{eq:omega2}
		\underline \omega_2 = \underline \omega_1 + \omega_1 u_1,
	\ee
	The contribution to $A_2$ from the two sides of $u_1$ is given by 
	\[
		( e^{2\pi i \frac{z }{\underline \omega_1 \times u_1 } }| e^{2\pi i \frac{\omega_1}{\underline \omega_1 \times u_1 }} , e^{2\pi i \frac{\underline \omega_1 \times u_0}{\underline \omega_1 \times u_1 }} )_\infty
		( e^{2\pi i \frac{z +\underline \omega_2 \times u_2 }{\underline \omega_2 \times u_1 } }| e^{2\pi i \frac{\omega_1}{\underline \omega_2 \times u_1 }} , e^{2\pi i \frac{\underline \omega_2 \times u_2}{\underline \omega_2 \times u_1 }} )_\infty . 
	\]
	Using periodicity and the functional equations that the q-factorial enjoys as well as the relations \eqref{eq:sumofnormals}, \eqref{eq:omega2}  and \eqref{eq:omegaparamrel}, we see that this is equal to 
	\begin{align}
		&( e^{2\pi i \frac{z}{\underline \omega_1 \times u_1 } }| e^{2\pi i \frac{\omega_1}{\underline \omega_1 \times u_1 }} , e^{2\pi i \frac{\underline \omega_1 \times u_0}{\underline \omega_1 \times u_1 }} )_\infty
		 ( e^{2\pi i \frac{z }{\underline \omega_1 \times u_1 } }| e^{2\pi i \frac{\omega_1}{\underline \omega_1 \times u_1 }} , e^{-2\pi i \frac{\underline \omega_2 \times u_2}{\underline \omega_1 \times u_1 }} ) ^{-1}_\infty \nn \\
		=& ( e^{2\pi i \frac{z}{\underline \omega_1 \times u_1 } }| e^{2\pi i \frac{\omega_1}{\underline \omega_1 \times u_1 }} , e^{2\pi i \frac{\underline \omega_1 \times u_0}{\underline \omega_1 \times u_1 }} )_\infty
		 ( e^{2\pi i \frac{z }{\underline \omega_1 \times u_1 } }| e^{2\pi i \frac{\omega_1}{\underline \omega_1 \times u_1 }} , e^{-2\pi i \frac{\underline \omega_2 \times ( a u_1 - u_0 ) }{\underline \omega_1 \times u_1 }} )^{-1}_\infty \nn		\\
		=& ( e^{2\pi i \frac{z}{\underline \omega_1 \times u_1 } }| e^{2\pi i \frac{\omega_1}{\underline \omega_1 \times u_1 }} , e^{2\pi i \frac{\underline \omega_1 \times u_0}{\underline \omega_1 \times u_1 }} )_\infty
		( e^{2\pi i \frac{z }{\underline \omega_1 \times u_1 } }| e^{2\pi i \frac{\omega_1}{\underline \omega_1 \times u_1 }} , e^{2\pi i \frac{\underline \omega_1  \times  u_0  + \omega_1 }{\underline \omega_1 \times u_1 }} )^{-1}_\infty \nn . 
	\end{align}
	Moreover, we use proposition \ref{prop:gluing} which allows us to combine this into the single q-shifted factorial
	\[
			( e^{ 2\pi i \frac{z}{\underline \omega_1 \times u_1} } | e^{ 2\pi i \frac{\underline \omega_2 \times u_2 }{\underline\omega_1 \times u_1 }}, e^{2\pi i \frac{\underline \omega_1 \times u_0}{\underline \omega_1 \times u_1} } )_\infty , 
	\]	
	where we also used \eqref{eq:omega2}. 
	Note that if $v_1 = (1,-L_1)$, $v_2 = (1,-L_2)$ it follows from $u_1 = L_2 - L_1$ that 
	\[
		\det [ v_1,v_2,\underline \omega ] = \omega_1 ( L_1 \times L_2) + \omega_2 ( L_2^2 - L_1^3) + \omega_3 ( L_1^2 - L_2^2) = \underline \omega_1 \times u_1.  
	\]
	Also note that if we let $n=(0,-u_2)$, then 
	\[
		\det [v_1,v_2, n] = \det [ v_1-v_2,v_2,n] = \begin{pmatrix}
				0 & 1 & 0 \\
				u_1 & -L_2 & -u_2
			\end{pmatrix} = -u_2 \times u_1 = 1.
	\]
	and then it is straight forward to check that indeed
	\begin{align}
		\det [ n, \underline \omega, v_2 ] &= \underline \omega_2 \times u_2 , \\
		 \det [ v_1, \underline \omega, n] &= -\underline \omega_1\times u_2 = \underline \omega_1\times u_0 + a \underline \omega_1 \times u_1 , 
	\end{align}
	where the ambiguity in the integer $a$ doesn't matter since it will only enter as $e^{2\pi ia}$. Comparing this with the prescription given in the statement of the theorem, it is a straight forward exercise to see that the above parameters indeed match.

	Finally, we consider the factor $B' = B e^{-i\pi \frac{z}{\omega_1} + \frac{i \pi}{2}}$. The logarithm of this is given by 
		\begin{align*}
		\log B' &= -\frac{\pi i }{6} \sum_{i=1}^N \sum_{j=0}^{N_i}  B_{33} (z + \underline \omega_i \times u_{i,j}  | \omega_1,  \underline \omega_i \times u_{i,j}, \underline \omega_i \times u_{i,j+1} ) -i\pi \frac{z}{\omega_1}  + \frac{i \pi}{2} . 
		\end{align*}
	By following a procedure analogous to the one used in $\zeta_r^C$ case, it is possible to resolve the sum in equation \eqref{eq:genBernoulliDef}, using the 1-Gorenstein condition as above, and find that the above is indeed equal to $B_{3,3}^C ( z | \underline \omega )$. This completes the proof. 		
\end{proof}

The two above results for $S_2^C$ and $S_3^C$ reduces to the previously known infinite product representations of $S_2$ and $S_3$ if one chooses the standard cones $\mathbb{R}^2_{\geq 0}$ or $\mathbb{R}^3_{\geq 0}$, and the action of the $SL_{r+1}(\mathbb{Z})$ on the parameters of the multiple q-factorials is essentially the same as the group action given in equation \eqref{eq:groupaction}. In the next section, we will show a very similar factorization property for $G_1^C$ and $G_2^C$. 

\section{Generalized multiple elliptic gamma functions }\label{sec:generalizedgamma}
\begin{definition} Let $C$ be a good cone in $\mathbb{R}^{r+1}$, and assume that $\mathrm{Im}(\underline \omega) \in (C^*)^\circ$. We define the \emph{generalized multiple elliptic gamma function} associated to $C$ as 
\be
	G_r^C ( z | \underline \tau ) = \prod_{n \in C \cap \mathbb{Z}^r }( 1 - e^{2\pi i ( z + n\cdot \underline \omega )} )^{(-1)^r} \prod_{n\in C^{\circ} \cap \mathbb{Z}^r} ( 1 - e^{2\pi i (- z + n \cdot \underline \omega )}). 
\ee
\end{definition}
This definition closely mimics one way of writing the usual multiple elliptic gamma functions, as shown in equation \eqref{eq:GrDef2}, and if one chooses the standard cone $C = \mathbb{R}^r_{\geq 0}$ the function is exactly the ordinary $G_r$. Notice that the usual definition uses the q-shifted factorials, which has a natural extension allowing $\omega_i \in \mathbb{C} - \mathbb{R}$, whereas here we require that $\mathrm{Im} (\underline \omega)$ is strictly inside the dual of the cone, so that the product above converges. 

For $r=1$, this definition is essentially the same as the gamma functions associated to wedges defined in \cite{FHRZ}. 
\subsection{Factorization property of $G_1^C$}

\begin{theorem}[Factorization property of $G_1^C$] \label{thm:G1Cfactorization}
	Let $C$ be a 2-dimensional good cone defined by the two normals $\{ v_1, v_2 \}$. Then the associated generalized elliptic gamma function can be written as
	\be
		G_1^C ( z | \underline \omega ) = e^{ \frac{\pi i }{3} B_{3,3}^{\tilde C} (z | \underline \omega , -1) } \prod_{f=1}^2 (SK_f)^{*} G_1( z | \underline \omega ) , 
	\ee
	where $S, K_f$ are defined in section \ref{sec:ConesModularity}, and where $g\in SL_{3} (\mathbb{Z})$ acts on $G_1$ by the group action given in  equation \eqref{eq:groupaction} on its parameters. $\tilde C$ is a 3d cone with normals $(v_1,0),(v_2,0),(0,0,1)$, and $B_{3,3}^{\tilde C}$ is its associated Bernoulli polynomial. 	
\end{theorem}
\begin{proof}
This proof needs again the subdivision of the wedge as constructed in lemma \ref{lem:subdivision}, but now in order to perform an infinite product rather than a sum. The factorization of the original infinite product into products over smaller wedges is well defined. To see this,  we just need the obvious generalization of equation \eqref{eq:plethystic1}, that the logarithm of the product gives a sum which converges absolutely, which allows us to rearrange the product in this way. The product over each subdivided wedge will converge, and after an $SL_2(\mathbb{Z})$ transformation gives a q-factorial,  
\be
	 \prod_{n \in C \cap \mathbb{Z}^2 }( 1 - e^{2\pi i ( z + n\cdot \underline \omega )} ) = (e^{2\pi iz} | e^{2\pi i(\underline \omega \times u_n)}, e^{2\pi i(\underline \omega \times u_{n+1})})_\infty \prod_{j=0}^{n-1} ( e^{2\pi i(z + \underline \omega \times u_j)} | e^{2\pi i(\underline \omega \times u_j)}, e^{2\pi i(\underline \omega \times u_{j+1})} )_\infty .
\ee
Using this, and the alternative choice for including lines and taking care to only include the interior of the cone for the other infinite product, it follows that 
\be
	G_1^C ( z | \underline \omega ) = G_1 ( z | \underline \omega \times u_n, \underline \omega \times u_{n+1}) \prod_{j=0}^{n-1} G_1( z + \underline \omega \times u_j | \underline \omega \times u_j, \underline \omega \times u_{j+1} ).
\ee	
Next, we apply the result of theorem \ref{thm:G2modularity}, in order to write this as 
\be
\begin{split}
	G_1^C ( z | \underline \omega ) = B \times G_1 ( \frac{z}{\underline \omega \times u_n} | -\frac{1}{\underline \omega \times u_n}, \frac{u_{n+1}}{\underline \omega \times u_n} ) G_1 ( \frac{z}{\underline \omega \times u_{n+1}} | -\frac{1}{\underline \omega \times u_{n+1}},\frac{u_{n}}{\underline \omega \times u_{n+1}} )\\
	\times \prod_{j=0}^{n-1} G_1 ( \frac{z + \underline \omega \times u_n}{\underline \omega \times u_n} | -\frac{1}{\underline \omega \times u_n}, \frac{u_{n+1}}{\underline \omega \times u_n} ) G_1 ( \frac{z+ \underline \omega \times u_n}{\underline \omega \times u_{n+1}} | -\frac{1}{\underline \omega \times u_{n+1}},\frac{u_{n}}{\underline \omega \times u_{n+1}} )		, 
	\end{split}
\ee
where $B$ is the exponential of a sum of $B_{3,3}$ terms. A quick computation starting from the definition of $B_{r,n}^C$ and using the same procedure as in the proof of theorem \ref{thm:S3Cfactorization} gives 
\be
	B = e^{\frac{\pi i}{3} B_{3,3}^{\tilde C} ( z | (\underline \omega , -1))}.
\ee
To deal with the product over $G_1$'s, one considers the product over two factors with the same denominator in their arguments. Then, expressing each $G_1$ as a product of two q-factorials and matching the q-factorials up appropriately, the calculation seen in equation \eqref{eq:blockcancellation} goes through again and shows that all paired up $G_1$'s indeed cancel. So in the end we are left with only two remaining $G_1$-functions, 
\[
	G_1 ( \frac{z}{\underline \omega \times u_{n+1}} | -\frac{1}{\underline \omega \times u_{n+1}},\frac{\underline \omega \times u_{n}}{\underline \omega \times u_{n+1}} ) G_1 ( \frac{z}{\underline \omega \times u_{0}} | -\frac{1}{\underline \omega \times u_{0}},\frac{\underline \omega \times u_{1}}{\underline \omega \times u_{0}} ).
\]
Repeating the arguments from the end of the proof of theorem \ref{thm:S2Cfactorization}, one can easily compare this with the prescription of the theorem, and up to ordering, which of course do not matter since $G_r(z|\underline \omega)$ is symmetric under permutation of $\underline \omega$, we see that it precisely matches the group action of $S K_F$ on $(z|(\underline \omega,1))$ where the final 1 is removed.

\end{proof}

\subsection{Factorization property of $G_2^C$}
%We prove a similar representation for $G_2^C$, and just as was the case for $S_3^C$, we now have to require the cone $C$ to be 1-Gorenstein. Again, this is for technical reasons and we believe that the result should hold for any good cone. 
\begin{theorem}[Factorization property of $G_2^C$]
\label{thm:G2Cfactorization}
	Let $C$ be a good 3-dimensional cone satisfying the 1-Gorenstein condition with normals $\{ v_i \}_{i=1}^N$, and a set of 1d faces $\Delta_1^C$. Then 
	\be \label{eq:G2Cfactorization1}
		G_2^C ( z | \underline \omega ) = e^{\frac{\pi i }{12} B_{4,4}^{\tilde C} ( z | \underline \omega,-1) } \prod_{f \in \Delta^C_1} (SK_F)^* G_2 (z |\underline \omega ) .
	\ee
	where  $\tilde C$ is the 4d cone with normals $\{ (v_1,0), \ldots, (v_N,0), (0,0,0,1) \}$, $S, K_f \in SL_4 (\mathbb{Z})$ are defined as in section \ref{sec:ConesModularity} and $(SK_f)$ acts on $G_2$ by acting on its parameters $(z | \underline \omega ) $ as specified by equation \eqref{eq:groupaction}. 
\end{theorem}
\begin{proof}
The proof proceeds along the same lines as the proof of theorem \ref{thm:S3Cfactorization}, for this reason we will not write it in as much detail.

We use the 1-Gorenstein condition to let all the normals have the form $v_i = (1 , -L_i )$, i.e. choose $\xi = (1,0,0)$, and then consider the product over each face separately. 
At this point it is possible to split each face into the subdivided wedges of lemma \ref{lem:subdivision}, so that we get a set of subdivided normals $\{ u_{i,j} \}$, where each small wedge has normals forming a $SL_2(\mathbb{Z})$ basis. The product over each of these wedges can then be transformed into a product over $\mathbb{Z}^3_{\geq 0}$ through a shift and a $SL_2(\mathbb{Z})$ transformation. Doing this, and remembering to include the contribution from $n_2=n_3=0$, we get
\begin{align}
	G_2^C ( z | \underline \omega ) =&  \prod_{i=1}^N \prod_{j=0}^{N_i} \prod_{n\in \mathbb{Z}^3_{\geq 0}} ( 1 - e^{2\pi i ( z + \underline \omega_i \times u_{i,j} + n_1 \omega_1 +  n_2 \underline \omega_i \times u_{i,j} + n_3 \underline \omega_i \times u_{i,j+1} ) } ) \nn  \\
	&\times( 1 - e^{2\pi i ( \omega_1 - z + \underline \omega_i \times u_{i,j+1} + n_1 \omega_1 +  n_2 \underline \omega_i \times u_{i,j} + n_3 \underline \omega_i \times u_{i,j+1} ) } ) \times \big ( \prod_{n_1 = 0 }^\infty ( 1 - e^{2\pi i ( z + \omega_1 n_1 ) } ) ( 1 - e^{2\pi i (\omega_1 -  z + \omega_1 n_1 ) } ) \big ) \nn.
\end{align}
All these factors combine into ordinary $G_2$ functions except for the last product which combines into $G_0 ( z | \omega_1)$ 
 \be
 	G_2^C ( z | \underline \omega ) =  G_0 ( z | \omega_1 ) \prod_{i=1}^N \prod_{j=0}^{N_i} G_2 ( z + \underline \omega_i \times u_{i,j} | \omega_1, \underline \omega_ i \times u_{i,j}, \underline \omega_i \times u_{i,j+1} ) .
 \ee	
We apply the first version of theorem \ref{thm:G2modularity}, i.e. the modular property of $G_2$, to rewrite this as 
	\begin{align} \label{eq:usingG2mod}
		&G_2^C ( z | \underline \omega ) =   G_0 ( z | \omega_1 ) \times B \times \overbrace{\prod_{i=1}^N \prod_{j=0}^{N_i} G_2 ( \frac{z}{\omega_1} | -\frac{1}{\omega_1} , \frac{ \underline \omega_i \times u_{i,j} }{\omega_1}, \frac{ \underline \omega_i \times u_{i,j+1} }{\omega_1} )}^{A_1} \times 		
		\nn \\
		& \underbrace{ G_2 ( \frac{z}{ \underline \omega_i \times u_{i,j}} | \frac{\omega_1}{ \underline \omega_i \times u_{i,j}} , - \frac{1}{ \underline \omega_i \times u_{i,j}}, \frac{ \underline \omega_i \times u_{i,j+1}}{ \underline \omega_i \times u_{i,j}} ) 
		G_2 ( \frac{z}{ \underline \omega_i \times u_{i,j+1}} | \frac{\omega_1}{ \underline \omega_i \times u_{i,j+1}} ,  \frac{ \underline \omega_i \times u_{i,j}}{ \underline \omega_i \times u_{i,j+1}},- \frac{1}{ \underline \omega_i \times u_{i,j+1}} ))}_{A_2},
	\end{align}
	where
	\[
	 	B = \prod_{i=1}^N \prod_{j=0}^{N_i} e^{\frac{\pi i }{12} B_{4,4} ( z + \underline \omega_i \times u_{i,j} | \omega_1, \underline \omega_ i \times u_{i,j}, \underline \omega_i \times u_{i,j+1} ,-1 ) }.
	\]
	Now, using that $u_{i,j-1} + u_{i,j+1} = \mathbb{Z} u_{i,j}$, and that $\underline \omega_{i+1} = \underline \omega_i + u_{i,N_i+1}$ as shown in the previous proof, and properties of $G_2$, one shows that every factor of $A_2$ cancels, except for the crossing from one face of the cone to another, just as in the proof of the factorization of $S_3^C$. The contributions from the two factors on either side of a crossing between two wedges can be written as 
	\[
		\frac{G_2 ( \frac{z} { \underline \omega_i \times u_{i,N_i + 1}} | \frac{\omega_1} { \underline \omega_i \times u_{i,N_i + 1}}, \frac{\underline \omega_i \times u_{i+1,1}} { \underline \omega_i \times u_{i,N_i + 1}}, -\frac{1} { \underline \omega_i \times u_{i,N_i + 1} } )}{ G_2 ( \frac{z} { \underline \omega_i \times u_{i,N_i + 1}} | \frac{\omega_1} { \underline \omega_i \times u_{i,N_i + 1}}, \frac{\underline \omega_i \times u_{i+1,1}} { \underline \omega_i \times u_{i,N_i + 1}}, -\frac{1} { \underline \omega_i \times u_{i,N_i + 1} } ) } , 
	\]
	and then using proposition \ref{prop:G2gluing}, the above combines into 
	\[
		G_2 ( \frac{z}{\underline \omega_i \times u_{i,N_i+1} } | \frac{\underline \omega_{i+1} \times u_{i+1,1} } {\underline \omega_i \times u_{i,N_i+1}  }, \frac{-\underline \omega_i \times u_{i+1,1} } {\underline \omega_i \times u_{i,N_i+1} }, - \frac 1 {\underline \omega_i \times u_{i,N_i+1} } ) .
	\]
	So from the $A_2$ factor, we get one such factor from each intersection of two faces, and we see that the parameters are the same as in the proof of theorem \ref{thm:S3Cfactorization}. So we can write
	\[
	 	A_2 = \prod_{f\in\Delta_1^C} (SK_f)^*G_2 ( z | \underline \omega ).
	\]
	
	Next, by means of a slight generalization of corollary \ref{cor:A1cancellation} one can show that $A_1 = G_0( \frac{z}{\omega_1} | - \frac{1}{\omega_1} )$.	
	The two $G_0$ functions we now have combine into an exponential of a polynomial in $z$ due to the following modular property of $G_0 = \theta_0$ :
	\be
		\theta_0 \left ( \frac{z}{\tau} , -\frac{1}{\tau} \right )= e^{
-\pi i \,B_{2,2} (z|(\underline{\tau},-1))} \theta_0 \left (z,\tau \right )
	\ee
	which combines with the $B$-factor into what we will call $B'$. One can then compute the sum defining $B_{4,4}^{\tilde C}$ by subdividing the 3d 1-Gorenstein cone as in this proof and then performing the sum over the extra direction. Doing this, one finds that indeed
	\be
		B' = e^{\frac{\pi i}{12} B_{4,4}^{\tilde C} ( z | \underline \omega )}. 
	\ee
	This concludes the proof.

\end{proof}

This and the previous result for $G_1^C$ includes all the previously known modularity results about the elliptic double gamma function $\Gamma$ investigated by Felder and Varchenko \cite{FV,GF_FVKZB} as well as the modularity results for $G_r$ proved by Narukawa in \cite{Narukawa:2003} as special cases for the regular cones. It is also a generalization of the work \cite{FHRZ}, and it could have some interpretation in terms of higher dimensional generalizations of the gerbe defined there. In ongoing work \cite{future_work}, we extend this result for all $r$, using slightly different methods than in the present article.

\begin{corollary} \label{cor:altG2Cfact}
	There is an alternative factorization of $G_2^C$ as 
	\be \label{eq:G2Cfactorization2}
		G_2^C( z | \underline \omega) = e^{ -\frac{\pi i}{12} B_{4,4}^{\tilde C} ( z | \underline \omega, 1) } \prod_{f \in \Delta_1^C} (S^{-1} K_f )^* G_2 ( z|\underline \omega ).
	\ee 
	where $C$ is a good, 1-Gorenstein 3d cone.
\end{corollary}
\begin{proof}
This is simply a consequence of the alternative factorization of the usual $G_2$ in theorem \ref{thm:G2modularity}, which we insert into the proof of theorem \ref{thm:G2Cfactorization} in equation \eqref{eq:usingG2mod}, and then the proof goes through with very minor changes. The parameters change sign, and one realizes that we can precisely account for this by replacing the $S$-duality matrix with its inverse $S^{-1}$. Finally, a somewhat tedious but straightforward calculation is needed to see that the new Bernoulli factors one gets is exactly $B_{4,4}^{\tilde C} ( z | (\underline \omega, 1))$, as we would expect from comparison with normal $G_2$.  
\end{proof}

Corollary \ref{cor:altG2Cfact} leads to the following result, which resembles the modularity property of the normal $G_r$ functions  in theorem \ref{thm:G2modularity2}
\begin{proposition}
	\be
		%\exp [-\frac{\pi i}{12} ( B_{4,4}^{\tilde C} ( z | (\underline \omega,-1)) + B_{4,4}^{\tilde C} ( z | (\underline \omega,1)))]
		\exp(-\frac{\pi i }{3} B_{3,3}^C ( z | \underline \omega ) ) = \prod_{f\in \Delta_1^C} (S K_f)^* G_1 ( z | \underline \omega ).
	\ee
	where $C$ is a good, 1-Gorenstein 3d cone and we abuse the notation slightly with respect to equation \eqref{eq:groupaction} since we are neglecting the action on the first component. The action of $SK_f$ on $(z|\underline \omega)$ is now taken to be
	\be
		g\cdot ( z | \underline \omega ) = \Bigl(\frac{z}{(g\underline \omega)_4} | \frac{(g\underline \omega)_2 }{(g\underline \omega,)_4} ,\frac{(g\underline \omega)_3 }{(g\underline \omega)_4}  \Bigr). 
	\ee
\end{proposition}
\begin{proof}
	Using equations \eqref{eq:G2Cfactorization1} and \eqref{eq:G2Cfactorization2} we can factorize $G_2^C$  in two equivalent ways,
%\bea G_2^C(z|\underline \omega)&=&B^+\frac{G_0(z|\omega_1)}{G_0(\frac{z}{\omega_1}|\frac{-1}{\omega_1})}\prod_{i=1}^{N}G_2\big(\gb_iz|\gb_i\ep_i,-\gb_i\ep'_i,\gb_i\big)\nn\\
%&=&B^-\frac{G_0(z|\omega_1)}{G_0(\frac{-z}{\omega_1}|\frac{-1}{\omega_1})}\prod_{i=1}^{N}G_2\big(-\gb_iz|-\gb_i\ep_i,-\gb_i\ep'_i,-\gb_i\big),\nn\eea
%The Bernoulli factors can be expressed as:
%\bea B^{\pm}=\exp\frac{\pm 2\pi i}{4!}B_{4,4}(z|\underline \omega,\mp1),\nn\eea
and if we divide the two different expressions for the factorization we get:
\begin{align*}
1 &= \exp \big [\frac{\pi i}{12} ( B_{4,4}^{\tilde C} ( z | (\underline \omega,-1)) + B_{4,4}^{\tilde C} ( z | (\underline \omega,1)))\big]
\prod_{f\in\Delta_1^C} \frac{ (SK_f)^* G_2 ( z | \underline \omega ) } { (S^{-1}K_f)^* G_2 ( z | \underline \omega ) } . %\\
%&= \exp \big [\frac{\pi i}{12} ( B_{4,4}^{\tilde C} ( z | (\underline \omega,-1)) + B_{4,4}^{\tilde C} ( z | (\underline \omega,1)))\big] 
%\prod_{f\in\Delta_1^C} \frac{ (SK_f)^* G_2 ( z | \underline \omega ) } { (S^{-1}K_f)^* G_2 ( z | \underline \omega ) } \\
%&=
%\frac{\big(\gb_iz|\gb_i\ep_i,-\gb_i\ep'_i,\gb_i\big)}{\big(-\gb_iz|-\gb_i\ep_i,-\gb_i\ep'_i,\gb_i\big)}\cdotp\frac{\big(\gb_iz|\gb_i\ep_i,\gb_i\ep'_i,\gb_i\big)}{\big(-\gb_iz|-\gb_i\ep_i,-\gb_i\ep'_i,-\gb_i\big)}\nn\\
%&=  \exp \big [ \frac{\pi i}{12} ( B_{4,4}^{\tilde C} ( z | (\underline \omega,-1)) + B_{4,4}^{\tilde C} ( z | (\underline \omega,1)))\big] 
%\prod_{i=1}^{N}\frac{\big(\gb_iz|\gb_i\ep_i,\gb_i\ep'_i\big)}{\big(-\gb_iz|-\gb_i\ep_i,-\gb_i\ep'_i\big)}\nn\\
%&= \exp \big [ \frac{\pi i}{12} ( B_{4,4}^{\tilde C} ( z | (\underline \omega,-1)) + B_{4,4}^{\tilde C} ( z | (\underline \omega,1)))\big]
%\prod_{i=1}^{N}G_1\big(\gb_iz|\gb_i\ep_i,\gb_i\ep'_i\big).\nn
\end{align*}
Consider the factor of $G_2$'s at a particular face $f$. Explicitly, if $(SK_f)\cdot ( z | \underline \omega) = ( \frac{z}{\tau} | -\frac{1}{\tau},\frac{\sigma}{\tau},\frac{\epsilon}{\tau})$, then $(S^{-1}K_f)\cdot ( z | \underline \omega) = ( -\frac{z}{\tau} | -\frac{1}{\tau},-\frac{\sigma}{\tau},-\frac{\epsilon}{\tau})$, and we have 
\[
	\frac{ G_2 ( \frac{z}{\tau} | -\frac{1}{\tau},\frac{\sigma}{\tau},\frac{\epsilon}{\tau}) } { G_2 ( -\frac{z}{\tau} | -\frac{1}{\tau},-\frac{\sigma}{\tau},-\frac{\epsilon}{\tau}) } .
\]
Expanding each $G_2$ into two q-factorials according to the definition, and using repeatedly the properties in equation \eqref{eq:blockinversion} and the definition of $G_1$, it is a short computation to see that this is equal to 
\be
	\frac{1}{G_1 \left ( \frac{z}{\tau} | \frac{\sigma}{\tau},\frac{\epsilon}{\tau} \right ) }  = \frac{1}{(SK_f)^* G_1 ( z | \underline \omega)},
\ee
with the abuse of notation explained above. 
Next, using a property of Bernoulli polynomials \cite{Narukawa:2003}:
\be
	B_{r,r} ( z | (\underline \omega, \eta ) ) + B_{r,r} ( z | (\underline \omega, -\eta ) ) = - r B_{r-1,r-1} ( z | \underline \omega ), 
\ee
together with the expression for $B_{4,4}^{\tilde C}$ as a sum of ordinary Bernoulli polynomials, and comparing with the expression for $B_{3,3}^C$ as a sum over $B_{3,3}$'s, and doing a short computation for the few additional terms, we see that indeed
\be
	\exp \big [ \frac{\pi i}{12} ( B_{4,4}^{\tilde C} ( z | (\underline \omega,-1)) + B_{4,4}^{\tilde C} ( z | (\underline \omega,1)))\big] = \exp [ -\frac{\pi i}{3} B_{3,3}^C ( z | \underline \omega ) ] 
\ee
%Combining the prefactor gives:
%\bea \log\frac{B^+}{B^-}=2\pi i\big(-\frac{4z^3}{3\omega_1}+2z^2-\frac{2}{3}\omega_1z\big)\sum_{i=1}^{N}\frac{A_i}{|\vec v_i|}+2\pi i\big(-\frac{z}{12}+\frac{\omega_1}{24}\big)\frac{1}{2\pi}\sum_{i=1}^{N}\gb_i-2\pi i\big(-\frac{z}{\omega_1}+\frac{1}{2}\big).\nn\eea
\end{proof}

\end{document}